\documentclass[a4paper,11pt,oneside,fleqn]{amsart}

\usepackage{amsrefs, amsmath, amsthm, amssymb}
\usepackage[marginratio=1:1,height=632pt,width=360pt,tmargin=117pt]{geometry}
\usepackage{booktabs}
\usepackage{wasysym}

\usepackage{color}
\usepackage{colortbl}
\definecolor{LightGray}{gray}{0.8}

\usepackage{subfig}

\usepackage{graphicx}
\graphicspath{{./Figures/}}

\usepackage{chngpage}

\usepackage{floatpag}

\usepackage{tikz}

\def\PS{
\begin{tikzpicture}[scale = 1.7]
\draw[line width = 0.5] (-0.1, 0.0) -- (0.1,0.0);
\draw[line width = 0.5] (-0.1, 0.0) -- (0.0,0.1732);
\draw[line width = 0.5] ( 0.1, 0.0) -- (0.0,0.1732);
\end{tikzpicture}
}

\def\PSsmall{
\begin{tikzpicture}[scale = 1.3]
\draw[line width = 0.5] (-0.1, 0.0) -- (0.1,0.0);
\draw[line width = 0.5] (-0.1, 0.0) -- (0.0,0.1732);
\draw[line width = 0.5] ( 0.1, 0.0) -- (0.0,0.1732);
\end{tikzpicture}
}

\def\PSB{
\begin{tikzpicture}[scale = 1.7]
\draw[line width = 0.5] (-0.1, 0.0) -- (0.1,0.0);
\draw[line width = 0.5] (-0.1, 0.0) -- (0.0,0.1732);
\draw[line width = 0.5] ( 0.1, 0.0) -- (0.0,0.1732);

\draw[very thin]  (-0.05, 0.5*0.1732) -- (0.05, 0.5*0.1732);
\draw[very thin]  (-0.05, 0.5*0.1732) -- (0.0, 0.0);
\draw[very thin]  (0.0, 0.0) -- (0.05, 0.5*0.1732);

\draw[very thin]  (0.0, 0.0) -- (0.0,0.1732);
\draw[very thin] (-0.1, 0.0) -- (0.05, 0.5*0.1732);
\draw[very thin] (-0.05, 0.5*0.1732) -- (0.1,0.0);
\end{tikzpicture}
}

\def\PSD{
\begin{tikzpicture}[baseline={([yshift=-.5ex]current bounding box.center)}, scale = 1.7]
\draw[very thin]  (-0.1, 0.1732) -- (0.1, 0.1732);
\draw[very thin]  (-0.1, 0.1732) -- (0.0, 0.0);
\draw[very thin]  (0.0, 0.0) -- (0.1, 0.1732);

\draw[very thin]  (0.0, 0.0) -- (0.0, 0.1732);
\draw[very thin] (-0.05, 0.5*0.1732) -- (0.1, 0.1732);
\draw[very thin] (-0.1, 0.1732) -- (0.05, 0.5*0.1732);
\end{tikzpicture}
}

\def\PSE{
\begin{tikzpicture}[scale = 1.7]
\draw[line width = 0.5] (-0.1, 0.0) -- (0.1,0.0);
\draw[line width = 0.5] (-0.1, 0.0) -- (0.0,0.1732);
\draw[line width = 0.5] ( 0.1, 0.0) -- (0.0,0.1732);

\draw[very thin]  (0.0, 0.0) -- (0.0,0.1732);
\end{tikzpicture}
}

\def\PSF{
\begin{tikzpicture}[scale = 1.7]
\draw[line width = 0.5] (-0.1, 0.0) -- (0.1,0.0);
\draw[line width = 0.5] (-0.1, 0.0) -- (0.0,0.1732);
\draw[line width = 0.5] ( 0.1, 0.0) -- (0.0,0.1732);

\draw[very thin] (-0.1, 0.0) -- (0.05, 0.5*0.1732);
\end{tikzpicture}
}

\def\PSG{\begin{tikzpicture}[scale = 1.7]
\draw[line width = 0.5] (-0.1, 0.0) -- (0.1,0.0);
\draw[line width = 0.5] (-0.1, 0.0) -- (0.0,0.1732);
\draw[line width = 0.5] ( 0.1, 0.0) -- (0.0,0.1732);

\draw[very thin] (-0.05, 0.5*0.1732) -- (0.1,0.0);
\end{tikzpicture}}

\newcommand{\SimS}[1]{\raisebox{-0.75em}{\includegraphics[scale=0.27]{SimplexSplineSmall#1.pdf}}}
\newcommand{\SimSgeneric}{\raisebox{-0.75em}{\includegraphics[scale=0.27]{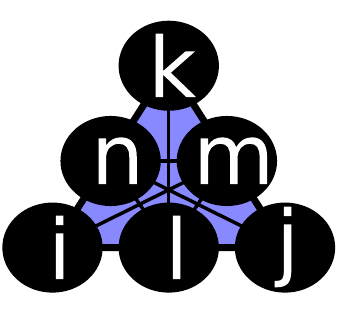}}}

\def\RR{\mathbb{R}}
\def\ZZ{\mathbb{Z}}

\def\BBB{\mathcal{B}}
\def\SSS{\mathcal{S}}
\def\TTT{\mathcal{T}}
\def\VVV{\mathcal{V}}
\def\EEE{\mathcal{E}}
\def\FFF{\mathcal{F}}

\def\bfc{\boldsymbol{c}}
\def\bff{\boldsymbol{f}}
\def\bfm{\boldsymbol{m}}
\def\bfp{\boldsymbol{p}}
\def\bfq{\boldsymbol{q}}
\def\bfu{{\boldsymbol{u}}}
\def\bfv{\boldsymbol{v}}
\def\bfx{\boldsymbol{x}}
\def\bfy{\boldsymbol{y}}
\def\bfone{\boldsymbol{1}}
\def\bfxi{ {\boldsymbol{\xi}} }

\def\bfS{\boldsymbol{S}}
\def\bfP{\mathbf{P}}
\def\bfK{\mathbf{K}}
\def\bfX{\mathbf{X}}
\def\bfY{\mathbf{Y}}

\def\mT{\text{T}}
\def\area{\text{area}}

\newtheorem{theorem}{Theorem}
\newtheorem{corollary}{Corollary}
\newtheorem{lemma}{Lemma}
\theoremstyle{definition}
\newtheorem{definition}{Definition}
\newtheorem{example}{Example}
\theoremstyle{remark}
\newtheorem{remark}[theorem]{Remark}

\title[]{Stable simplex spline bases\\ for $C^3$ quintics on the Powell-Sabin 12-split}

\author{Tom Lyche}
\address{University of Oslo, Department of Mathematics, P.O. Box 1053, Blindern, NO-0316, Oslo, Norway}
\email{tom@math.uio.no}

\author{Georg Muntingh}
\address{SINTEF ICT, Department of Applied Mathematics, P.O. Box 124 Blindern, NO-0314, Oslo, Norway}
\email{georg.muntingh@sintef.no}

\begin{document}
\begin{abstract}
For the space of $C^3$ quintics on the Powell-Sabin 12-split of a triangle, we determine explicitly the six symmetric simplex spline bases that reduce to a B-spline basis on each edge, have a positive partition of unity, a Marsden identity that splits into real linear factors, and an intuitive domain mesh. The bases are stable in the $L_\infty$ norm with a condition number independent of the geometry, have a well-conditioned Lagrange interpolant at the domain points, and a quasi-interpolant with local approximation order 6. We show an $h^2$ bound for the distance between the control points and the values of a spline at the corresponding domain points. For one of these bases we derive $C^0$, $C^1$, $C^2$ and $C^3$ conditions on the control points of two splines on adjacent macrotriangles.
\end{abstract}

\maketitle

\section{Introduction}
\noindent Piecewise polynomials or splines defined over triangulations form an indispensable tool in the sciences, with applications ranging from scattered data fitting to finding numerical solutions to partial differential equations. See \cite{Lai.Schumaker07,Ciarlet.78} for comprehensive monographs.

In applications like geometric modelling \cite{Cohen.Riensenfeld.Elber01} and solving PDEs by isogeometric methods \cite{Hughesbook} one often desires a low degree spline with $C^1$, $C^2$ or $C^3$ smoothness. For a general triangulation, it was shown in Theorem~1.(ii) of \cite{Zenisek.74} that the minimal degree of a triangular $C^r$ element is $4r + 1$, e.g., degrees $5,9,13$ for the classes $C^1$, $C^2$, $C^3$. To obtain smooth splines of lower degree one can split each triangle in the triangulation into several subtriangles. One such split is the Powell-Sabin 12-split $\PSB$ of a triangle $\PS$; see Figure \ref{fig:PS12}. On this split global $C^1$ smoothness can be obtained with degree only 2 \cite{Powell.Sabin77}, and $C^2$ smoothness with degree only $5$  \cite{Lai.Schumaker03,Lyche.Muntingh14, Schumaker.Sorokina06} on any (planar) triangulation.

Once a space is chosen one determines its dimension. The space $\SSS^1_2$ and $\SSS^3_5$ of $C^1$ quadratics and $C^3$ quintics on the 12-split of a single triangle have dimension 12 and 39, respectively. For a general triangulation $\TTT$ of a polygonal domain, we can replace each triangle in $\TTT$ by its 12-split to obtain a triangulation $\TTT_{12}$. The dimensions of the corresponding $C^1$ quadratic and $C^2$ quintic spaces (the latter with $C^3$ supersmoothness at the vertices of $\TTT$ and interior edges of the 12-split of each triangle in $\TTT$) are $3|\VVV| + |\EEE|$ and $10|\VVV| + 3|\EEE|$, where $|\VVV|$ and $|\EEE|$ are the number of vertices and edges in $\TTT$. Moreover, in addition to giving $C^1$ and $C^2$ spaces on any triangulation these spaces are suitable for multiresolution analysis \cite{Davydov.Yeo13, DynLyche98, Lyche.Muntingh14, Oswald92}.
 
For a general smoothness, degree, and triangulation, it is a hard problem to find a basis of the corresponding spline space with all the usual properties of the univariate B-spline basis. One reason is that it is difficult to find a single recipe for the various valences and topologies of the triangulations. In \cite{Cohen.Lyche.Riesenfeld13} a basis, called the (quadratic) S-basis, was constructed for the $C^1$ quadratics on the Powell-Sabin 12-split on one triangle. The S-basis consists of simplex splines \cite{Micchelli79, Prautsch.Boehm.Paluszny02} and has all the usual properties of univariate B-splines, including a recurrence relation down to piecewise linear polynomials and a Marsden identity. Moreover, analogous to the Bernstein-B\'ezier case, $C^0$ and $C^1$ smoothness conditions were given tying the S-bases on neighboring triangles together to give $C^1$ smoothness on the refined triangulation $\TTT_{12}$.

In the remainder of the next section we recall some background on splines and the 12-split. Section \ref{sec:DimensionFormula} introduces dimension formulas for spline spaces on the 12-split. In the next section some basic properties of simplex splines are recalled, after which an exhaustive list is derived of the $C^3$ quintic simplex splines on the 12-split that reduce to a B-spline on the boundary. Section \ref{sec:SimplexSplineBases} introduces a barycentric form of the Marsden identity and describes how the dual basis in \cite{Lyche.Muntingh14} can be applied to find the simplex spline bases of $\SSS^3_5$ satisfying a Marsden identity. Making use of the computer algebra system {\tt Sage} \cite{Sage}, these techniques are applied in Section~\ref{sec:Results} to discover six symmetric simplex spline bases that reduce to a B-spline basis on each boundary edge, have a positive partition of unity, a Marsden identity, and domain points with an intuitive control mesh and unisolvent for $\SSS^3_5$. The concise and coordinate independent form of the barycentric Marsden identity makes it possible to state these results in Table \ref{tab:DualPolynomials}. Analogous to the Bernstein-B\'ezier case, we find $C^0$, $C^1$, $C^2$ and even $C^3$ conditions for one of these bases, tying its splines together on a general triangulation. One of the latter smoothness conditions involves just the control points in a single triangle and the geometry of the neighboring triangles, showing that $C^3$ smoothness using $\SSS^3_5$ cannot be achieved on a general refined triangulation $\TTT_{12}$. Finally a conversion to the Hermite nodal basis from \cite{Lyche.Muntingh14} is provided.

\section{Background}
\subsection{Notation}
On a triangulation $\TTT$ of a polygonal domain $\Omega\subset\RR^2$ we define the spline spaces
\[ \SSS_d^r(\TTT):=\{f\in C^r(\Omega): f|_{\PSsmall} \in\Pi_d\ \forall \PS\in\TTT\}, \ r,d\in\ZZ,\  r\ge -1,\ d\ge 0, \]
where $\Pi_d$ is the space of bivariate polynomials of total degree at most $d$. With the convention $\Pi_d:=\emptyset$ and $\dim\Pi_d:=0$ if $d<0$, one has
\begin{equation} \label{eq:dim_Pi_k}
\dim \Pi_d = \frac12 (d+2)(d+1)_+,\qquad d\in \ZZ,
\end{equation}
where $z_+:=\max\{0,z\}$ for any real number $z$. 

Any point $\bfx$ in a nondegenerate triangle $[\bfv_1, \bfv_2, \bfv_3]$ can be represented by its \emph{barycentric coordinates} $(\beta_1, \beta_2, \beta_3)$, which are uniquely defined by $\bfx = \beta_1\bfv_1 + \beta_2 \bfv_2 + \beta_3\bfv_3$ and $\beta_1 + \beta_2 + \beta_3 = 1$. Similarly, each vector $\bfu$ is uniquely described by its \emph{directional coordinates}, i.e., the triple $(\beta_1 - \beta_1', \beta_2 - \beta_2', \beta_3 - \beta_3')$ with $(\beta_1, \beta_2, \beta_3)$ and $(\beta_1', \beta_2', \beta_3')$ the barycentric coordinates of two points $\bfx$ and $\bfx'$ such that $\bfu = \bfx - \bfx'$. Sometimes we write $\bfx$ as a linear combination of more than three vertices, in which case these coordinates are no longer unique.

A bivariate polynomial $p$ of total degree $d$ defined on a triangle $\PS\subset \RR^2$ is conveniently represented by its \emph{B\'ezier form}
\[
p(\bfx) = \sum_{\substack{i+j+k=d\\i,j,k\geq 0}} c_{ijk} B_{ijk}^d(\bfx),\qquad
B_{ijk}^d(\bfx) := \frac{d!}{i!j!k!}\beta_1^i \beta_2^j \beta_3^k,
\]
with $(\beta_1, \beta_2, \beta_3)$ the barycentric coordinates of $\bfx$ with respect to $\PS$. Here the $B_{ijk}^d$ are the \emph{Bernstein basis polynomials} of degree~$d$ and the coefficients $c_{ijk}$ are the \emph{B\'ezier ordinates} of $p$. We associate each B\'ezier ordinate $c_{ijk}\in \RR$ to the \emph{domain point} $\bfxi_{ijk} := \frac{i}{d}\bfv_1 + \frac{j}{d}\bfv_2 + \frac{k}{d}\bfv_3 \in \RR^2$ and combine them into the \emph{control point} $(\bfxi_{ijk}, c_{ijk}) \in \RR^3$. By connecting any two domain points $\bfxi_{i_1j_1k_1}$ and $\bfxi_{i_2j_2k_2}$ by a line segment whenever $|i_1 - i_2| + |j_1 - j_2| + |k_1 - k_2| = 1$, one arrives at the \emph{domain mesh} of $p$; see Figure \ref{fig:domain_mesh_Bezier}. The \emph{control mesh} is defined similarly by connecting control points.

\begin{figure}
\subfloat[]{\includegraphics[scale=0.67]{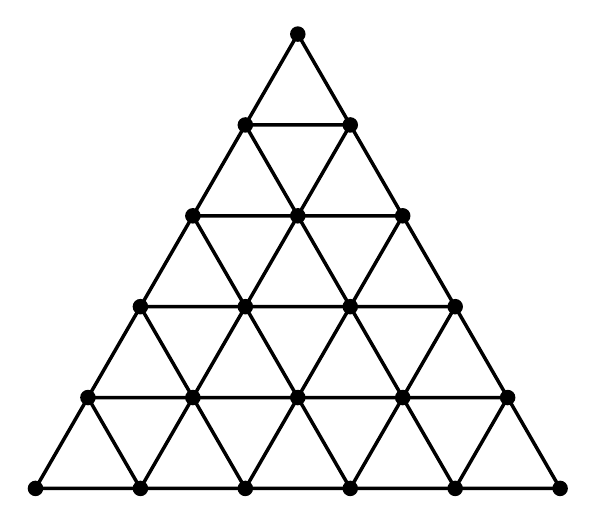}\label{fig:domain_mesh_Bezier}}
\subfloat[]{\includegraphics[scale=0.67]{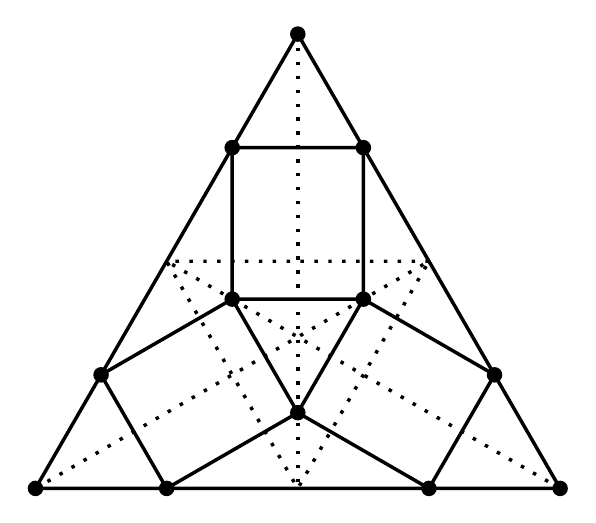}\label{fig:domain_mesh_S-basis}}
\subfloat[]{\includegraphics[scale=0.67]{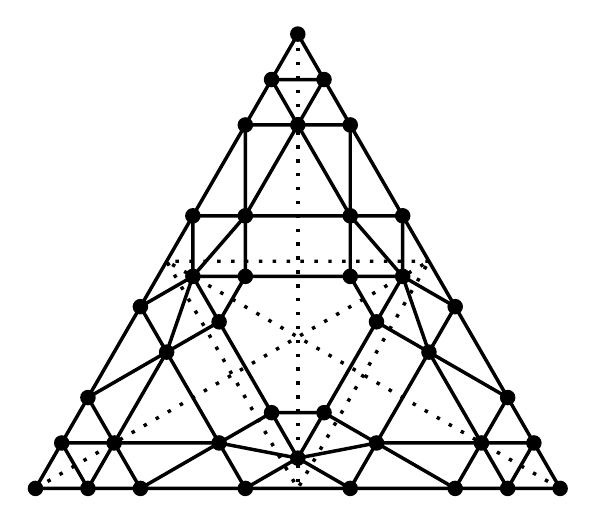}\label{fig:domain_mesh_Bc}}
\caption{Domain mesh for the quintic B\'ezier basis (left), the S-basis from \cite{Cohen.Lyche.Riesenfeld13} (middle) and the basis $\BBB_c$ in this paper (right).}\label{fig:domain_mesh}
\end{figure}

We consider finite multisets $\bfK = \{ \bfv_1^{m_1} \cdots \bfv_n^{m_n} \} \subset \RR^2$, in which the distinct elements $\bfv_1, \ldots, \bfv_n$ are counted with corresponding multiplicities $m_1,\ldots, m_n \geq 0$. Write $|\bfK|:= m_1 + \cdots + m_n$ for the total number of elements in $\bfK$. For any two integers $i,j$, the Kronecker delta is the symbol
\[ \delta_{ij} :=
\left\{\begin{array}{cl}
1 & \text{if } i = j,\\ 0 & \text{otherwise}.
\end{array}\right.\] 
For any set $A\subset \RR^2$, define the indicator function
\[
\bfone_A: \RR^2\longrightarrow \{0, 1\},\qquad
\bfone_A(x) := \left\{
\begin{array}{cl}
1 & \text{if~}x\in A,\\
0 & \text{if~}x\notin A.
\end{array}
\right.
\]

\subsection{The Powell-Sabin 12-split}
\noindent Given a triangle $\PS = [\bfv_1, \bfv_2, \bfv_3]$ with vertices $\bfv_1, \bfv_2, \bfv_3 $ write $e_1 := [\bfv_2, \bfv_3]$, $e_2 := [\bfv_3, \bfv_1]$, and $e_3 := [\bfv_1, \bfv_2]$ for its (nonoriented) edges. Connecting vertices and the edge midpoints $\bfv_4 := (\bfv_1 + \bfv_2)/2$, $\bfv_5 := (\bfv_2 + \bfv_3)/2$ and $\bfv_6 := (\bfv_1 + \bfv_3)/2$, we arrive at the \emph{Powell-Sabin 12-split} $\PSB$ of $\PS$; see Figure \ref{fig:PS12-Labels1} for the labelling of the vertices $\bfv_1, \ldots, \bfv_{10}$ and faces $\PS_1,\ldots,\PS_{12}$.

To decide to which face of the 12-split points on the interior edges belong, we follow the convention in \cite{Seidel92} shown in Figure \ref{fig:TrianglePartition}, which can be quickly computed by Algorithm 1.1 in \cite{Cohen.Lyche.Riesenfeld13}. If $\bfK$ is a multiset satisfying $\bfK \subset \{\bfv_1,\ldots,\bfv_{10}\}$ as sets, its convex hull $[\bfK]$ is a union of some of the faces of the 12-split, and we define the \emph{half-open convex hull} $[\bfK)$ as the union of the corresponding half-open faces depicted in Figure \ref{fig:TrianglePartition}. 

\begin{figure}
\begin{center}
\subfloat[]{\includegraphics[scale=0.8]{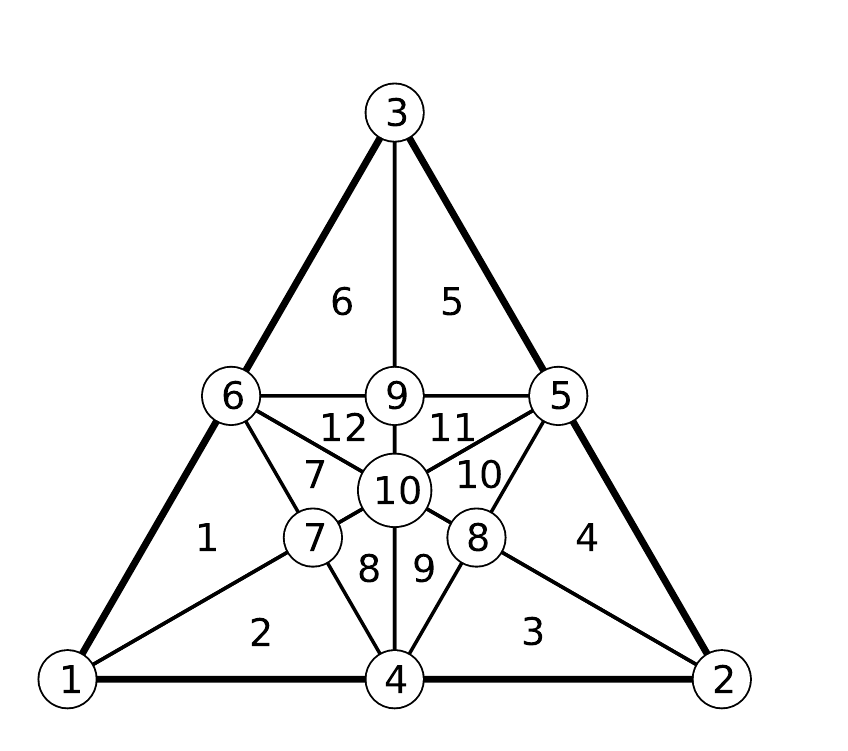}\label{fig:PS12-Labels1}}
\subfloat[]{\hspace{-2em}\includegraphics[scale=0.8]{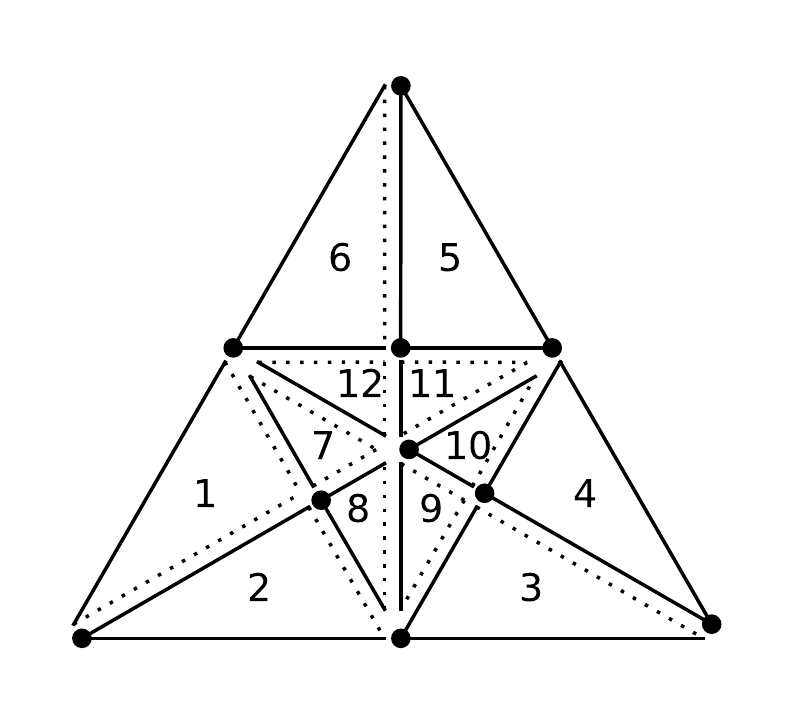}\label{fig:TrianglePartition}}
\end{center}
\caption{The Powell-Sabin 12-split with labelling of vertices and faces (left), and a scheme assigning every point in the macrotriangle to a unique face of the 12-split (right).}\label{fig:PS12}
\end{figure}

\begin{remark}
While in \cite{Zenisek.74} it was shown that on a general triangulation degree~9 is needed to achieve global $C^2$ smoothness, let us show that on the 12-split $\TTT_{12}$ of a triangulation the degree is necessarily at least $5$.
For $r = 0, 1, 2$, a $C^r$ spline of degree $r+2$ on the knot multiset $\{0^{r+3}\, 0.5^2\, 1^{r+3}\}$ is determined by $5+r$ degrees of freedom. To aim for $C^2$ smoothness using quartics, one uses the minimal number of degrees of freedom as follows. At each vertex fix derivatives up to order 2, and on each edge fix one additional value, two cross-boundary derivatives, and three second-order cross-boundary derivatives. Thus in total 36 degrees of freedom are needed, while the space $\SSS^2_4(\PSB)$ only has dimension 34; see Table~\ref{tab:dimensions}.
\end{remark}

\subsection{A basis for the dual space of a space of $C^3$ quintics}\label{sec:macroelement}
\noindent Let $\PSB$ be the 12-split of a triangle $\PS$ with vertices $\VVV = \{\bfv_1, \bfv_2, \bfv_3\}$ and edges $\EEE = \{[\bfv_1, \bfv_2]$, $[\bfv_2, \bfv_3], [\bfv_3, \bfv_1]\}$. For any edge $e = [\bfv_i, \bfv_j]\in \EEE$ with opposing vertex $\bfv_k$, let
\[ \bfq_{1,e} := \frac{3\bfv_i + \bfv_j}{4},\qquad \bfm_e := \frac{\bfv_i + \bfv_j}{2},\qquad \bfq_{2,e} := \frac{\bfv_i + 3\bfv_j}{4} \]
be its midpoint and quarterpoints. With $\varepsilon_{\bfv}$ the point evaluation at $\bfv$ and $D_\bfu$ the directional derivative in the direction $\bfu$, let
\begin{equation}\label{eq:Lambda}
\Lambda :=
   \bigcup_{\bfv\in \VVV}\ \bigcup_{\substack{i+j\leq 3\\i,j\geq 0}} \{\varepsilon_{\bfv} D^i_{\bfx_{\bfv}} D^j_{\bfy_{\bfv}} \} \cup
   \bigcup_{e\in \EEE} \{\varepsilon_{\bfq_{1,e}} D^2_{\bfu_e}, \varepsilon_{\bfm_e} D_{\bfu_e}, \varepsilon_{\bfq_{2,e}} D^2_{\bfu_e}\};
\end{equation}
see Figure~\ref{fig:PS12-1}. Here, for every vertex $\bfv$ and edge $e$, the symbols $\bfx_{\bfv}, \bfy_{\bfv}$ are any linearly independent vectors and $\bfu_e$ is any vector not tangent to $e$, for example the outside unit normal as shown in Figure \ref{fig:PS12-1}. It was shown in \cite[Theorem 4]{Lyche.Muntingh14} that $\Lambda$ is a basis for the dual space to $\SSS^3_5(\PSB)$, which therefore has dimension~39.  

\begin{figure} 
\begin{center}
\includegraphics[scale=0.69]{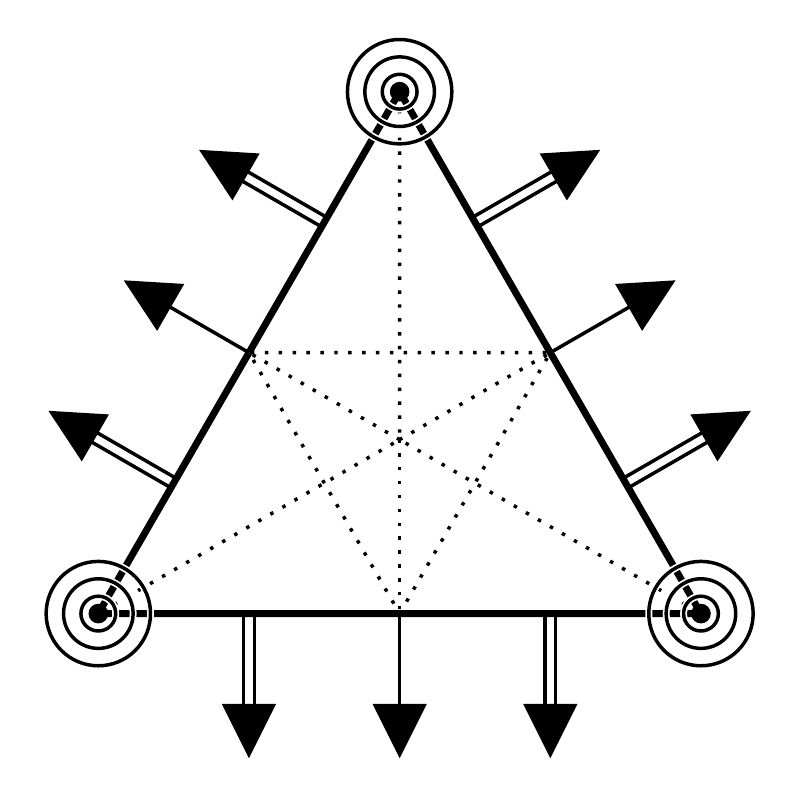}
\end{center}
\caption{A basis for the dual space of $\SSS^3_5$ on a single triangle. A bullet represents a point evaluation, three circles represent all derivatives up to order three, and a single and double arrow represent a first- and second-order directional derivative. These derivatives are evaluated at the rear end of the arrows, which are located at the midpoints and quarterpoints.}\label{fig:PS12-1}
\end{figure}

\section{Dimension formulas}\label{sec:DimensionFormula}
\noindent Consider a polygonal domain $\Omega\subset \RR^2$ with a triangulation $\TTT$ with sets of vertices $\VVV$, edges $\EEE$, faces $\FFF$, and 12-split refinement $\TTT_{12}$. Let
\[ \SSS^{2,3}_5(\TTT_{12}) := \{s\in \SSS^2_5(\TTT_{12}) : s\in C^3(\PS)\cap C^3(\bfv)\ \text{for all}\ \PS\in \FFF, \bfv \in \VVV\}.\]
Here $s\in C^3(\bfv)$ means that all polynomials $s|_{\PSsmall}$ such that $\PS$ is a triangle with vertex at $\bfv$ have common derivatives up to order three at the point $\bfv$. Note that if $\TTT$ consists of a single triangle, then $\SSS^{2,3}_5(\TTT_{12}) = \SSS^3_5(\TTT_{12})$.

Since $\Lambda$ from \eqref{eq:Lambda} specifies the value and partial derivatives up to order three at each vertex of $\TTT$ and the value, first- and second-order cross-boundary derivatives at each edge of $\TTT$, the following theorem is an immediate consequence of \cite[Theorem 4]{Lyche.Muntingh14}.

\begin{theorem}\label{thm:dimS}
For any triangulation $\TTT$ with $|\VVV|$ vertices and $|\EEE|$ edges, the set $\Lambda$ is a basis for the dual space of $\SSS^{2,3}_5(\TTT_{12})$. In particular
\begin{equation*}
\dim \SSS^{2,3}_5(\TTT_{12}) = 10|\VVV| + 3|\EEE|.
\end{equation*}
\end{theorem}

Next, let $\PSB$ denote the 12-split triangulation of a single triangle $\PS$. For future reference, we state the following formula for the dimension of the space of $C^r$ splines of degree $d$ on $\PSB$, which is a special case of Theorem 3.1 in \cite{Chui.Wang83} (also compare \cite{Schenck.Stillman97}).

\begin{theorem}\label{thm:DimensionFormula}
For any integers $d,r$ with $d\ge 0$ and $d\geq r\geq -1$,
\begin{equation}\label{eq:dimeq}
\begin{aligned}
\dim \SSS^r_d(\PSB) =\ & \frac12(r+1)(r+2) + \frac92(d-r)(d-r+1)\\
                       &+\frac32(d-2r-1)(d-2r)_+ + \sum_{j=1}^{d-r}(r-2j+1)_+.
\end{aligned}
\end{equation}
\end{theorem}

To quickly look up $\dim \SSS^r_d(\PSB)$ for small values of $r$ and $d$, we have listed these first dimensions in Table \ref{tab:dimensions}.

\begin{table}
\begin{tabular*}{\columnwidth}{c@{\extracolsep{\stretch{1}}}*{11}{r}}
\toprule
 $\dim \SSS^r_d (\PSB)$ & \ $C^{-1}$ & $C^0$ & $C^1$ & $C^2$ & $C^3$ & $C^4$ & $C^5$ & $C^6$ & $C^7$ & $C^8$ & $C^9$\\
\midrule
$d =  0$ &  12 &   1 \\
$d =  1$ &  36 &  10 &   3 \\
$d =  2$ &  72 &  31 &  12 &   6 \\
$d =  3$ & 120 &  64 &  30 &  16 &  10 \\
$d =  4$ & 180 & 109 &  60 &  34 &  21 &  15 \\
$d =  5$ & 252 & 166 & 102 &  61 & \cellcolor{LightGray}39 &  27 &  21 \\
$d =  6$ & 336 & 235 & 156 & 100 &  66 &  46 &  34 &  28 \\
$d =  7$ & 432 & 316 & 222 & 151 & 102 &  73 &  54 &  42 &  36 \\
$d =  8$ & 540 & 409 & 300 & 214 & 150 & 109 &  81 &  63 &  51 &  45 \\
$d =  9$ & 660 & 514 & 390 & 289 & 210 & 154 & 117 &  91 &  73 &  61 & 55 \\
\bottomrule
\end{tabular*}
\caption[]{Dimensions of $\SSS^r_d (\PSB)$, with $(r,d) = (3,5)$ highlighted.\label{tab:dimensions}}
\end{table}

\section{Simplex splines}\label{sec:simplexsplines}
\noindent In this section we first recall the definition and some basic properties of the simplex spline, and then proceed to determine the $C^3$ quintic simplex splines on the 12-split that reduce to a B-spline on the boundary. For a comprehensive account of the theory of simplex splines, see \cite{Micchelli79, Prautsch.Boehm.Paluszny02}.

\subsection{Definition and properties}
The following definition of the simplex spline is convenient for our purposes.

\begin{definition}
For any finite multiset $\bfK = \{ \bfv_1^{m_1} \cdots \bfv_{10}^{m_{10}} \} \subset \RR^2$ composed of vertices of $\PSB$, the \emph{(area normalized) simplex spline} $Q[\bfK]: \RR^2 \longrightarrow \RR$ is recursively defined by
\[
Q[\bfK](\bfx) :=
\left\{ \begin{array}{cl}
0 & \text{if~}\area([\bfK]) = 0,\\
\bfone_{[\bfK)}(\bfx)\frac{\area(\PSsmall)}{\area([\bfK])} & \text{if~}\area([\bfK]) \neq 0\text{~and~} |\bfK| = 3,\\
\sum_{j = 1}^{10} \beta_j Q[\bfK\backslash \bfv_j](\bfx) &  \text{if~}\area([\bfK]) \neq 0\text{~and~} |\bfK| > 3,\\
\end{array} \right. \]
with $\bfx = \beta_1 \bfv_1 + \cdots + \beta_{10} \bfv_{10}, \beta_1 + \cdots + \beta_{10} = 1$, and $\beta_i = 0$ whenever $m_i = 0$.
\end{definition}

By Theorem 4 in \cite{Micchelli79} this definition is independent of the choice of the $\beta_j$. It is well known that $Q[\bfK]$ is a piecewise polynomial with support $[\bfK] \subset \PS$ and of total degree at most $|\bfK| - 3$. One shows by induction on $|\bfK|$ that $\int_{\RR^2} Q[\bfK](\bfx) \text{d}\bfx = \area(\PS)\cdot {|\bfK| - 1\choose 2}^{-1}$. Although $M[\bfK] := \area(\PS)^{-1}{|\bfK|- 1\choose 2}Q[\bfK]$ has unit integral and is used more frequently, our discussion is simpler in terms of $Q[\bfK]$.

Whenever $m_7 = m_8 = m_9 = m_{10} = 0$, we use the graphical notation
\[ \SimSgeneric := Q[\bfv_1^i \bfv_2^j \bfv_3^k \bfv_4^l \bfv_5^m \bfv_6^n]. \]

\begin{example}\label{ex:S-basis}
The linear S-spline basis in \cite{Cohen.Lyche.Riesenfeld13} only uses vertices $\bfv_1,\ldots,\bfv_6,$ $\bfv_{10}$, while the quadratic S-spline basis only uses $\bfv_1,\ldots,\bfv_6$. It is given by
\begin{equation}\label{eq:S-splines}
S_{j,2} = \frac{\area([\bfK_{j,2}])}{6} M[\bfK_{j,2}] = \frac{\area([\bfK_{j,2}])}{\area(\PSsmall)}Q[\bfK_{j,2}],\hfill j=1,\ldots,12,
\end{equation}
where by (2.5) in \cite{Cohen.Lyche.Riesenfeld13}, as (unordered) sets,
\begin{equation}\label{eq:S-basis}
\begin{aligned}
\{Q[\bfK_{1,2}], \ldots, Q[\bfK_{12,2}]\} = \Big\{
& \SimS{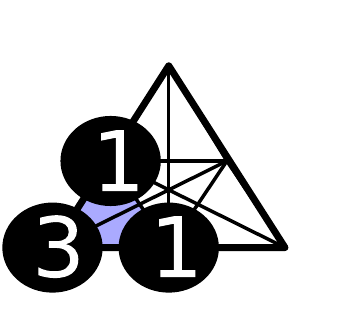}, \SimS{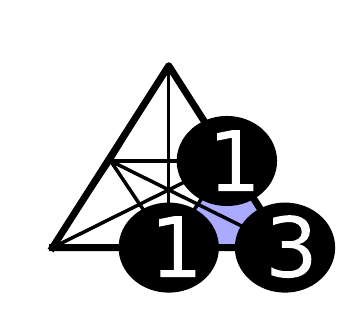}, \SimS{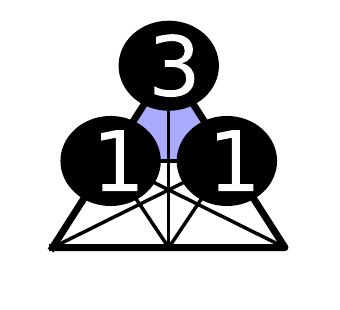}, \SimS{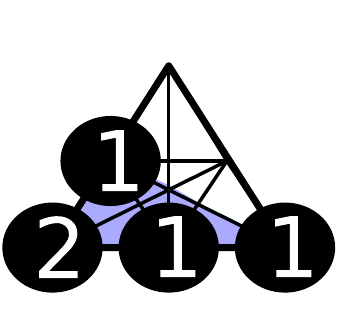}, \SimS{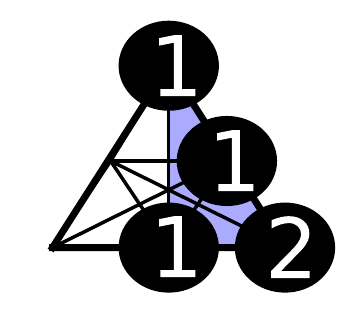}, \SimS{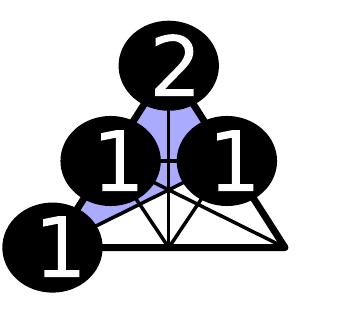}, \\
& \SimS{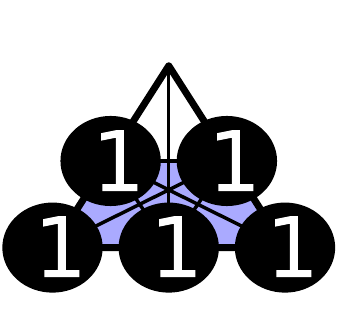}, \SimS{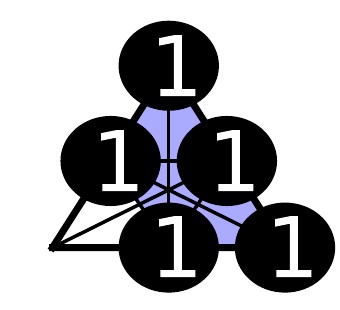}, \SimS{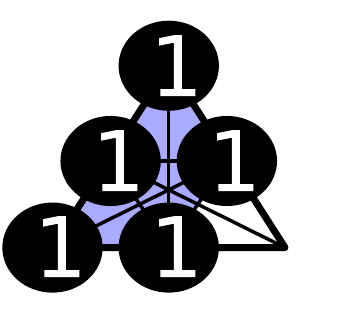}, \SimS{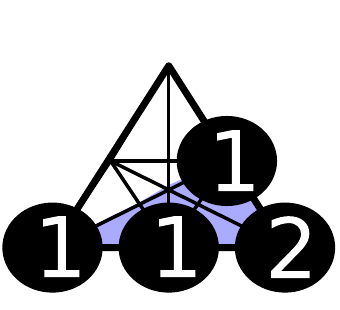}, \SimS{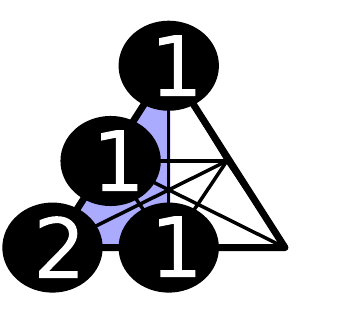}, \SimS{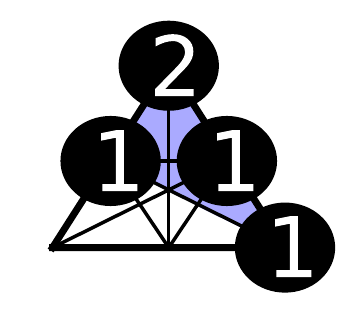}
\Big\}.
\end{aligned}
\end{equation}
\end{example}

\begin{example}
If $\bfK = \{\bfv_i^{\mu_i+1} \bfv_j^{\mu_j+1} \bfv_k^{\mu_k+1} \}$ has three distinct elements with $\area([\bfK]) > 0$, then, with $(\beta_i, \beta_j, \beta_k)$ the barycentric coordinates of $\bfx$ with respect to $[\bfv_i, \bfv_j, \bfv_k]$, it follows by induction that
\[
Q[\bfK](\bfx) = \bfone_{[\bfK)}(\bfx) \frac{\area(\PS)}{\area([\bfK])}\frac{(\mu_i + \mu_j + \mu_k)!}{\mu_i! \mu_j!\mu_k!}
\beta_i^{\mu_i} \beta_j^{\mu_j} \beta_k^{\mu_k}
\]
is, up to a scalar, a Bernstein polynomial on $[\bfK)$.
\end{example}

\subsubsection*{Continuity} For any edge $e$ of $\PSB$, if $\bfK$ has at most $m$ knots (counting multiplicities) along the affine hull of $e$, then $Q[\bfK]$ is $|\bfK| - m - 2$ times continuously differentiable over $e$. For instance, every $C^3$ quintic simplex spline on $\PSB$ has at most three knots along the affine hull of any interior edge $e$.

\subsubsection*{Differentiation}
Let $\bfK = \{\bfv_1^{m_1}\cdots \bfv_{10}^{m_{10}}\}$ be a finite multiset. If
$\bfu = \alpha_1 \bfv_1 + \cdots + \alpha_{10}\bfv_{10}$ is such that $\alpha_1 + \cdots + \alpha_{10} = 0$ and $\alpha_j = 0$ whenever $m_j = 0$, then one has a differentiation formula
\begin{equation}\label{eq:MicchelliDifferentiation}
D_\bfu Q[\bfK] = (|\bfK| - 3)\sum_{j = 1}^{10} \alpha_j Q[\bfK \backslash \bfv_j].
\end{equation}

\subsubsection*{Knot insertion}
If $\bfy = \beta_1 \bfv_1 + \cdots + \beta_{10}\bfv_{10}$ is such that $\beta_1 + \cdots + \beta_{10} = 1$ and $\beta_j = 0$ whenever $m_j = 0$, then one has a knot insertion formula
\begin{equation}\label{eq:knotinsertion}
Q[\bfK] = \sum_{j = 1}^{10} \beta_j Q[(\bfK \sqcup \bfy)\backslash \bfv_j].
\end{equation}
For instance, if $\bfv_1, \bfv_2 \in \bfK$, then, since $\bfv_4 = \frac12 \bfv_1 + \frac12 \bfv_2$, 
\begin{equation}
Q[\bfK] = \frac12 Q[(\bfK\sqcup \bfv_4) \backslash \bfv_1] + \frac12 Q[(\bfK\sqcup \bfv_4) \backslash \bfv_2],
\end{equation}
and for example
\[ \SimS{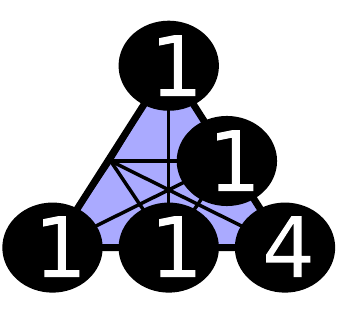} = \frac12 \SimS{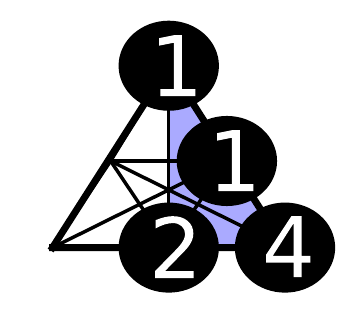} + \frac12 \SimS{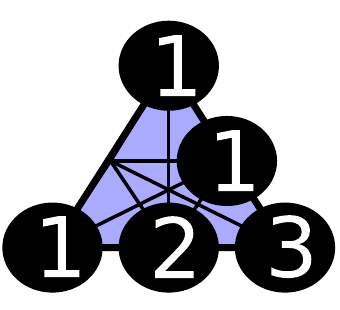}.\]

\subsubsection*{Restriction to an edge}
Let $e = [\bfv_i, \bfv_k]$ be an edge of $\PS$ with midpoint $\bfv_j$ and let $\varphi_{ik}(t) := (1-t)\bfv_i + t\bfv_k$. By induction on $|\bfK|$,
\begin{equation}\label{eq:restriction}
Q[\bfK] \circ \varphi_{ik}(t) =
\left\{ \begin{array}{cl}
                        0 & \text{if }m_i + m_j + m_k < |\bfK| - 1,\\
\frac{\area(\PSsmall)}{\area([\bfK])}B(t)  & \text{if }m_i + m_j + m_k = |\bfK| - 1,
\end{array} \right.
\end{equation}
where $B$ is the univariate B-spline with knot multiset $\{0^{m_i}\,0.5^{m_j}\,1^{m_k}\}$.

We say that $Q[\bfK]$ \emph{reduces to a B-spline on the boundary} when $B$ is one of the consecutive univariate quintic B-splines $B_1^5, \ldots, B_8^5$ on the open knot multiset $\{0^6\, 0.5^2\, 1^6\}$; see Table~\ref{tab:shorthandsBi}. Similarly a basis $\BBB = \{S_1, \ldots, S_{39}\}$ of $\SSS^3_5(\PSB)$ \emph{reduces to a B-spline basis on the boundary} when
\[ \big\{S_1 \circ \varphi_{ik}, \ldots, S_{39} \circ \varphi_{ik}\big\} = \{ \big(B^5_1 \big)^1\, \cdots\, \big(B^5_8\big)^1\, 0^{31} \},\qquad 1\leq i < k\leq 3,\]
as multisets. This scaling of $\BBB$ ensures simple $C^0$ conditions for connecting two adjacent patches expressed in terms of $\BBB$.

\subsubsection*{Symmetries}
\noindent The dihedral group $S_3$ of the equilateral triangle consists of the identity, two rotations and three reflections, i.e.,
\begin{center}
\includegraphics[scale=0.55]{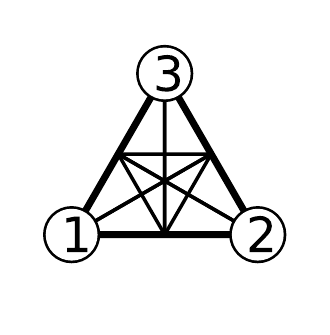}
\includegraphics[scale=0.55]{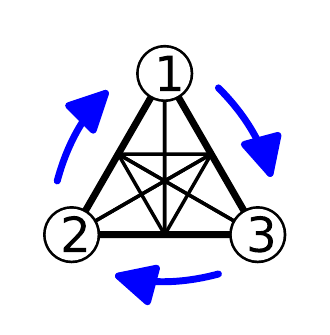}
\includegraphics[scale=0.55]{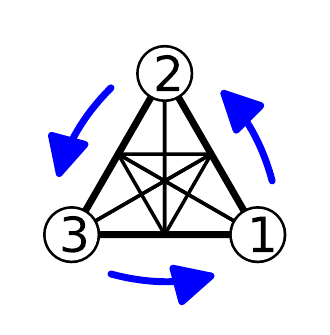}
\includegraphics[scale=0.55]{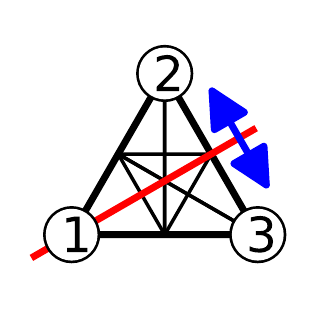}
\includegraphics[scale=0.55]{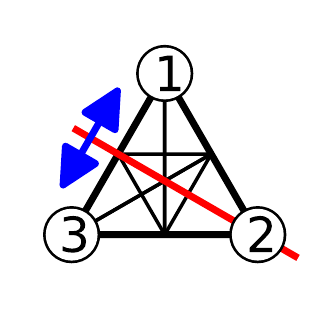}
\includegraphics[scale=0.55]{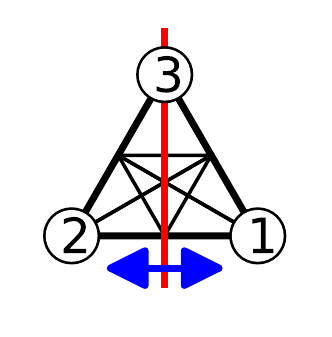}
\end{center}
The affine bijection sending $\bfv_k$ to $(\cos\,2\pi k/3, \sin\,2\pi k/3)$, for $k = 1, 2, 3$, maps $\PSB$ to the 12-split of an equilateral triangle. Through this correspondence, the dihedral group permutes the vertices $\bfv_1,\ldots,\bfv_{10}$ of $\PSB$. Every such permutation $\sigma$ induces a bijection $Q[\bfv_1^{m_1}\cdots\bfv_{10}^{m_{10}}] \longmapsto Q[\sigma(\bfv_1)^{m_1}\cdots $ $\sigma(\bfv_{10})^{m_{10}}]$ on the set of all simplex splines on $\PSB$. For any set $\BBB$ of simplex splines, we write
\[ [\BBB]_{S_3} := \{Q[\sigma(\bfK)]\,:\, Q[\bfK]\in \BBB,\ \sigma\in S_3\} \]
for the \emph{$S_3$ equivalence class of $\BBB$}, i.e., the \emph{set} of simplex splines related to $\BBB$ by a symmetry in $S_3$. For example,
\begin{align*}
\left[\SimS{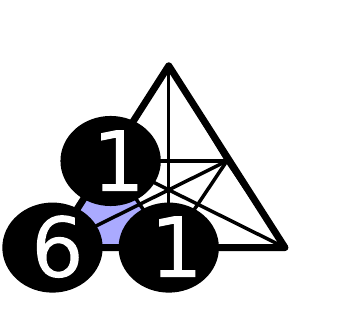}\right]_{S_3} & = \left\{\SimS{600101}, \SimS{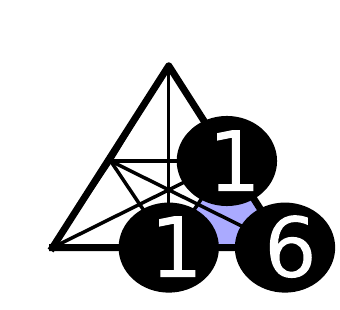}, \SimS{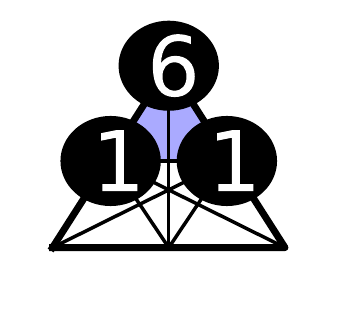}\right\},\\
\left[\SimS{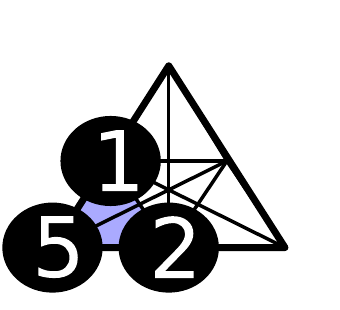}\right]_{S_3} & = \left\{\SimS{500201}, \SimS{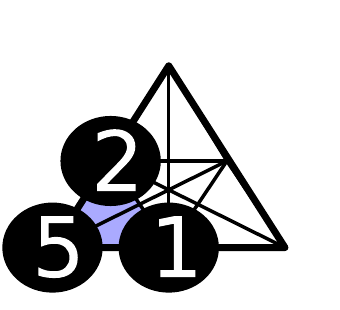}, \SimS{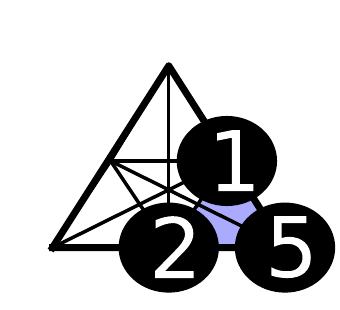}, \SimS{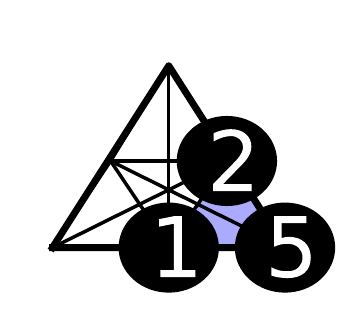}, \SimS{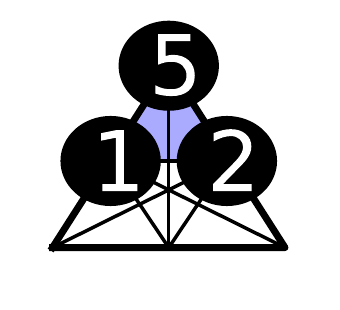}, \SimS{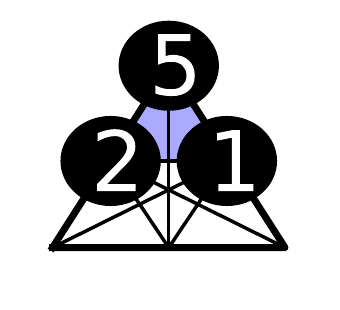}\right\},
\end{align*}
and \eqref{eq:S-basis} takes the compact form
\[ \left[\SimS{300101}, \SimS{210101}, \SimS{110111}\right]_{S_3}.\]
We say that $\BBB$ is \emph{$S_3$-invariant} whenever $[\BBB]_{S_3} = \BBB$.

\subsection{$C^3$ quintic simplex splines on the 12-split}
\noindent Any simplex spline $Q[\bfK]$ of degree $d=5$ on $\PSB$ is specified by a multiset $\bfK = \{\bfv_1^{m_1}\cdots \bfv_{10}^{m_{10}}\}$ satisfying
\begin{equation}\label{eq:MultiplicitySum}
m_1 + m_2 + \cdots + m_{10} = d + 3 = 8.
\end{equation}

\begin{lemma}
Suppose a quintic simplex spline $Q[\bfK]$ on $\PSB$ is of class $C^3$. Then
\begin{equation}\label{eq:WrongKnotLines2}
m_7 = m_8 = m_9 = m_{10} = 0,
\end{equation}
and
\begin{equation}\label{eq:InteriorSmoothness2}
m_1 + m_5,\ m_3 + m_4,\ m_2 + m_6,\ m_4 + m_6,\ m_4 + m_5,\ m_5 + m_6\leq 3,
\end{equation}
whenever both multiplicities are nonzero.
\end{lemma}

\begin{proof}
In the 12-split, certain knot lines do not appear, leading to the conditions
\begin{equation}\label{eq:WrongKnotLines}
\begin{aligned}
& m_1 m_8 = m_1 m_9 = m_8 m_9 = 0,\\
& m_2 m_7 = m_2 m_9 = m_7 m_9 = 0,\\
& m_3 m_7 = m_3 m_8 = m_7 m_8 = 0.
\end{aligned}
\end{equation}

To achieve $C^3$ smoothness over the remaining knot lines, the sum of the multiplicities along each line in the 12-split has to be at most three,
\begin{equation}
\begin{aligned}
 & m_1 + m_5 + m_7 + m_{10} \leq 3,\qquad m_4 + m_6 + m_7 \leq 3,\label{eq:InteriorSmoothness}\\
 & m_2 + m_6 + m_8 + m_{10} \leq 3,\qquad m_4 + m_5 + m_8 \leq 3,\\
 & m_3 + m_4 + m_9 + m_{10} \leq 3,\qquad m_5 + m_6 + m_9 \leq 3,
\end{aligned}
\end{equation}
whenever the line contains at least two knots with positive multiplicities.

Suppose $m_7 \geq 1$. Then \eqref{eq:WrongKnotLines} implies $m_2 = m_3 = m_8 = m_9 = 0$, and by \eqref{eq:MultiplicitySum} and \eqref{eq:InteriorSmoothness}, $8 = (m_1 + m_5 + m_7 + m_{10}) + (m_4 + m_6)\leq 3 + 2$, which is a contradiction. It follows that $m_7$ (and similarly $m_8$ and $m_9$) must be equal to zero. Moreover, by \eqref{eq:MultiplicitySum} and \eqref{eq:InteriorSmoothness},
$8 = (m_1 + m_5 + m_7 + m_{10}) + (m_2 + m_6 + m_8 + m_{10}) + (m_3 + m_4 + m_9 + m_{10}) - 2m_{10} \leq 9 - 2m_{10}$. Therefore $m_{10} = 0$ and \eqref{eq:WrongKnotLines2} holds, and \eqref{eq:InteriorSmoothness2} follows immediately from \eqref{eq:InteriorSmoothness}.
\end{proof}

In addition we demand that $Q[\bfK]$ reduces to a B-spline on the boundary. By \eqref {eq:restriction}, if $m_i + m_j + m_k < 7$, then $Q[\bfK]|_e = 0$ and this condition is satisfied. The remaining case $m_i + m_j + m_k = 7$ yields the conditions
\begin{equation}\label{eq:BoundaryBSpline}
\begin{aligned}
& \text{ not}(m_1 + m_4 + m_2 = 7 \text{ and } m_4 \geq 3),\\
& \text{ not}(m_1 + m_4 + m_2 = 7 \text{ and } m_1 \geq 1 \text{ and } m_2 \geq 1 \text{ and } m_4 \neq 2 ),\\
& \text{ not}(m_2 + m_5 + m_3 = 7 \text{ and } m_5 \geq 3),\\
& \text{ not}(m_2 + m_5 + m_3 = 7 \text{ and } m_2 \geq 1 \text{ and } m_3 \geq 1 \text{ and } m_5 \neq 2 ),\\
& \text{ not}(m_1 + m_6 + m_3 = 7 \text{ and } m_6 \geq 3),\\
& \text{ not}(m_1 + m_6 + m_3 = 7 \text{ and } m_1 \geq 1 \text{ and } m_3 \geq 1 \text{ and } m_6 \neq 2 ).
\end{aligned}
\end{equation}

\begin{table}
\begin{tabular}{cccccccccc}
\toprule
$Q^a$ & $Q^b$ & $Q^c$ & $Q^d$ & $Q^e$ & $Q^f$ & $Q^g$ & $Q^h$ & $Q^i$ & $Q^j$\\
   \SimS{600101} & \SimS{500201} & \SimS{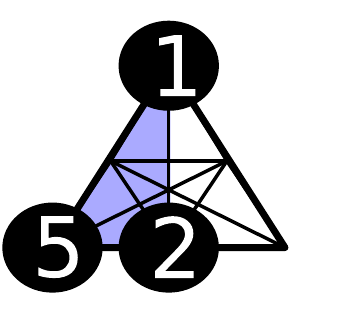} & \SimS{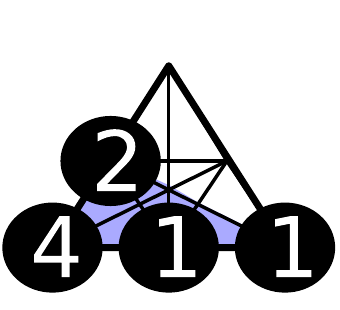} & \SimS{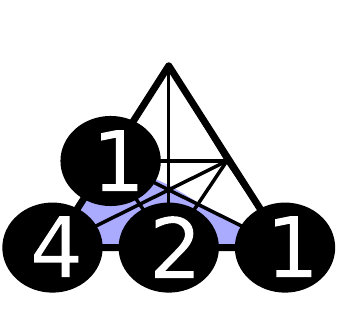}
 & \SimS{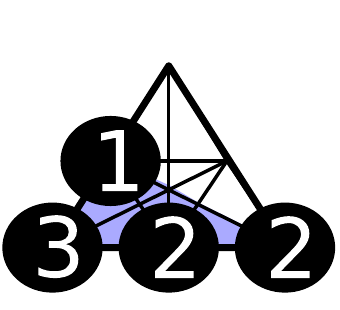} & \SimS{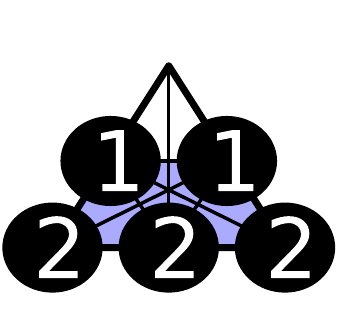} & \SimS{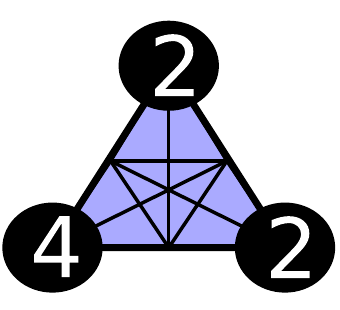} & \SimS{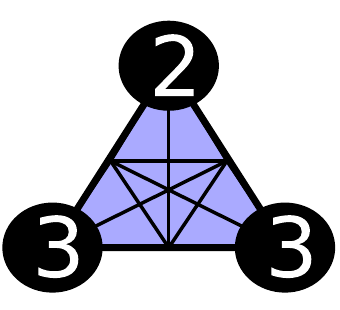} & \SimS{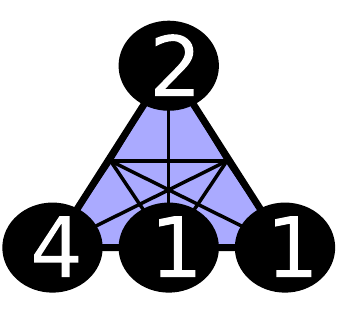}\\ \midrule
$Q^k$ & $Q^l$ & $Q^m$ & $Q^n$ & $Q^o$ & $Q^p$ & $Q^q$ & $Q^r$ & $Q^s$ & $Q^t$\\
   \SimS{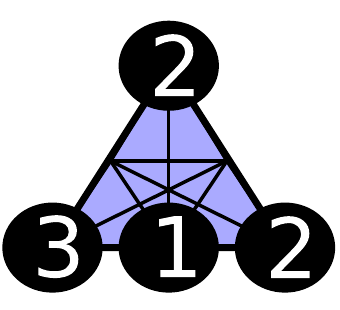} & \SimS{141110} & \SimS{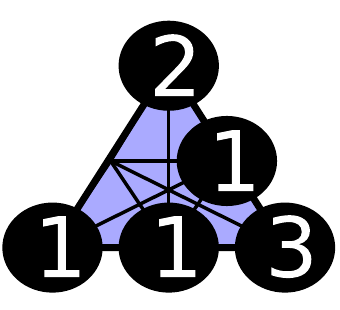} & \SimS{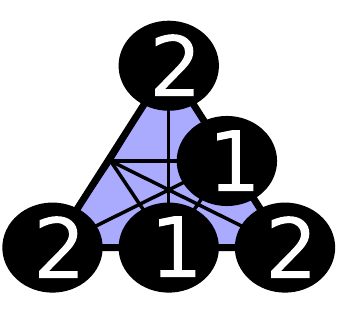} & \SimS{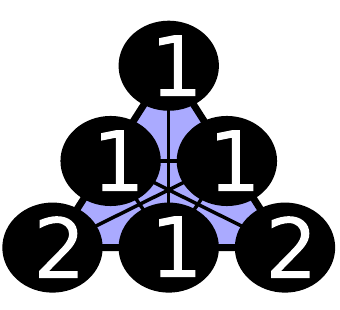}
 & \SimS{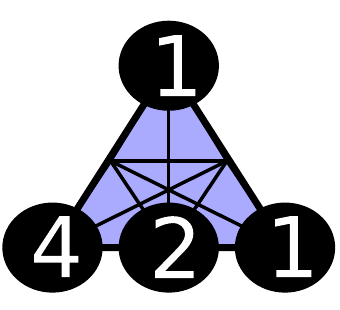} & \SimS{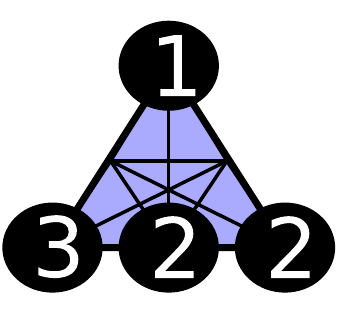} & \SimS{131210} & \SimS{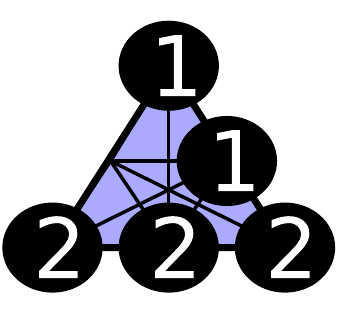} & \SimS{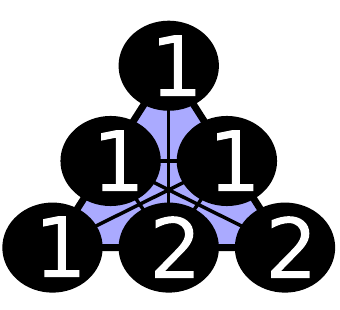}\\ \bottomrule
\end{tabular}
\caption[]{The $C^3$ quintic simplex splines on $\PSB$, one representative for each $S_3$ equivalence class, that reduce to a B-spline on the boundary.}\label{tab:AllSimplexSplines532}
\end{table}

Adding the entries in the second column of \eqref{eq:InteriorSmoothness} and using \eqref{eq:WrongKnotLines2} gives $2m_4 + 2m_5 + 2m_6 \leq 9$, implying, by \eqref{eq:MultiplicitySum},
\begin{equation}\label{eq:interiorexteriormultiplicities}
m_1 + m_2 + m_3 \geq 4,\qquad m_4 + m_5 + m_6 \leq 4.
\end{equation}

\begin{theorem}\label{thm:AllowedSimplexSplines}
With one representative for each $S_3$ equivalence class, Table~\ref{tab:AllSimplexSplines532} is an exhaustive list of the $C^3$ quintic simplex splines on $\PSB$ that reduce to a B-spline on the boundary.
\end{theorem}

\begin{proof}
Recall from \eqref{eq:WrongKnotLines2} that $m_7 = m_8 = m_9 = m_{10} = 0$. We distinguish cases according to the support $[\bfK]$ of $Q[\bfK]$, up to a symmetry in $S_3$.

\emph{Case 0, no corner included, $[\bfK] = [\bfv_4, \bfv_5, \bfv_6]$}: By \eqref{eq:MultiplicitySum}, $m_4 + m_5 + m_6 = 8$, so that the sum of two of these  multiplicities will be at least $5$, contradicting~\eqref{eq:InteriorSmoothness2}. Therefore this case does not occur.

\emph{Case 1a, 1 corner included, $[\bfK] = [\bfv_1, \bfv_4, \bfv_6]$}: For a positive support $m_1, m_4, m_6\geq 1$, and since $m_4 + m_6 \leq 3$ by \eqref{eq:InteriorSmoothness2}, we obtain 
\[ [Q^a]_{S_3} = \left[\SimS{600101}\right]_{S_3},\ [Q^b]_{S_3} = \left[\SimS{500201}\right]_{S_3}.\]

\emph{Case 1b, 1 corner included, $[\bfK] = [\bfv_1, \bfv_4, \bfv_5, \bfv_6]$}: 
By \eqref{eq:MultiplicitySum} and \eqref{eq:interiorexteriormultiplicities} one has $m_1 = 8 - m_4 - m_5 - m_6 \geq 4$, contradicting $m_1 + m_5 \leq 3$ from \eqref{eq:InteriorSmoothness2}. Therefore this case does not occur.

\emph{Case 2a, 2 corners included, $[\bfK] = [\bfv_1, \bfv_2, \bfv_6]$}: If $m_4 \in \{0,1\}$, then $m_6 \geq 2$ by the second line in \eqref{eq:BoundaryBSpline}, and since $m_6 \leq 3 - m_2$ by \eqref{eq:InteriorSmoothness2}, it follows that $m_2 = 1$ and $m_6 = 2$. We obtain
\[ \left[Q^c\right]_{S_3} = \left[\SimS{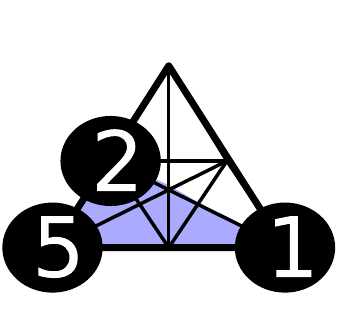}\right]_{S_3},\ [Q^d]_{S_3} = \left[\SimS{410102}\right]_{S_3}.\]
If $m_4 \geq 2$, then, since $m_4 + m_6 \leq 3$ by \eqref{eq:InteriorSmoothness2}, one has $m_4 = 2$ and $m_6 = 1$. Since $m_2\leq 3 - m_6 = 2$ by \eqref{eq:InteriorSmoothness2}, we obtain
\[ [Q^e]_{S_3} = \left[\SimS{410201}\right]_{S_3},\ [Q^f]_{S_3} = \left[\SimS{320201}\right]_{S_3}.\]

\emph{Case 2b, 2 corners included, $[\bfK] = [\bfv_1, \bfv_2, \bfv_5, \bfv_6]$}: Since $m_5 + m_6\leq 3$ by \eqref{eq:InteriorSmoothness2}, one has $1\leq m_5, m_6\leq 2$. Suppose $m_6 = 2$. Then $m_2 = m_5 = 1$ and $m_4 \leq 1$ by \eqref{eq:InteriorSmoothness2}, and $m_1 + m_5 = 8 - m_2 - m_4 - m_6 \geq 4$, contradicting \eqref{eq:InteriorSmoothness2}. We conclude that $m_6 = 1$ and similarly that $m_5 = 1$. Then $m_1 + m_2 + m_4 = 8 - m_5 - m_6 = 6$, and since \eqref{eq:InteriorSmoothness2} implies $m_1, m_2, m_4\leq 2$, one obtains
\[ [Q^g]_{S_3} = \left[\SimS{220211}\right]_{S_3}.\]

\emph{Case 3, 3 corners included, $[\bfK] = [\bfv_1, \bfv_2, \bfv_3]$}: We distinguish cases for $(m_4, m_5, m_6)$, with $m_4\geq m_5\geq m_6$, first by $m_4$, and then by $m_4 + m_5 + m_6$, which is at most 4 by~\eqref{eq:interiorexteriormultiplicities}.

\begin{itemize}
\item[(0,0,0)] One has $m_1, m_2, m_3\geq 2$ by \eqref{eq:BoundaryBSpline}, and we obtain
\[ [Q^h]_{S_3} = \left[\SimS{422000}\right]_{S_3},\ [Q^i]_{S_3} = \left[\SimS{332000}\right]_{S_3}.\]
\item[(1,0,0)] One has $2\leq m_3 \leq 3 - m_4$ by \eqref{eq:BoundaryBSpline} and \eqref{eq:InteriorSmoothness2}, yielding
\[ [Q^j]_{S_3} = \left[\SimS{412100}\right]_{S_3},\ [Q^k]_{S_3} = \left[\SimS{322100}\right]_{S_3}.\]
\item[(1,1,0)] One has $m_1, m_3 \leq 2$ by \eqref{eq:InteriorSmoothness2}, yielding
\[ [Q^l]_{S_3} = \left[\SimS{141110}\right]_{S_3},\ [Q^m]_{S_3} = \left[\SimS{132110}\right]_{S_3},\ [Q^n]_{S_3} = \left[\SimS{222110}\right]_{S_3}.\]
\item[(1,1,1)] One has $m_1, m_2, m_3 \leq 2$ by \eqref{eq:InteriorSmoothness2}, and we obtain
\[ [Q^o]_{S_3} = \left[\SimS{221111}\right]_{S_3}.\]
\item[(2,0,0)] One has $m_3 = 1$ by \eqref{eq:InteriorSmoothness2}, yielding
\[ [Q^p]_{S_3} = \left[\SimS{411200}\right]_{S_3},\ [Q^q]_{S_3} = \left[\SimS{321200}\right]_{S_3}.\]
\item[(2,1,0)] One has $m_3 = 1$ and $m_1 \leq 2$ by \eqref{eq:InteriorSmoothness2}, yielding
\[ [Q^r]_{S_3} = \left[\SimS{131210}\right]_{S_3},\ [Q^s]_{S_3} = \left[\SimS{221210}\right]_{S_3}.\]
\item[(2,1,1)] One has $m_3 = 1$ and $m_1, m_2\leq 2$ by \eqref{eq:InteriorSmoothness2}, yielding
\[ [Q^t]_{S_3} = \left[\SimS{121211}\right]_{S_3}.\hfill \qedhere \]
\end{itemize}
\end{proof}

\section{Simplex spline bases for $\SSS^3_5$}\label{sec:SimplexSplineBases}
\noindent Let $\Lambda$ as in \eqref{eq:Lambda} be a basis of the dual space of $\SSS^3_5(\PSB)$. In this section we describe a recipe for determining the $S_3$-invariant simplex spline bases that reduce to a B-spline basis on the boundary, having a positive partition of unity and a Marsden identity with only real linear factors.

\subsection{Potential bases}
\noindent Of the $C^3$ quintic simplex splines in Table~\ref{tab:AllSimplexSplines532}, only
\begin{center}
\begin{tabular}{ccccccc}
$Q^a$ & $Q^b$ & $Q^c$ & $Q^e$ & $Q^f$ & $Q^p$ & $Q^q$\\
\SimS{600101}, & \SimS{500201}, & \SimS{501200}, & \SimS{410201}, & \SimS{320201}, & \SimS{411200}, & \SimS{321200}
\end{tabular}
\end{center}
are nonzero on $[\bfv_1, \bfv_2]$. Any $S_3$-invariant basis reducing to a B-spline basis on the boundary should therefore contain the $S_3$ equivalence class of one of the eight possible combinations in the Cartesian product
\[\left\{\SimS{600101}\right\} \times \left\{\SimS{500201}, \SimS{501200}\right\}\times \left\{\SimS{410201},\SimS{411200}\right\}\times\left\{\SimS{320201},\SimS{321200}\right\}. \]
This determines $3 + 6 + 6 + 6 = 21$ of the $39$ basis elements.

Of the 13 remaining $S_3$ equivalence classes in Table~\ref{tab:AllSimplexSplines532} of simplex splines that are zero on $[\bfv_1, \bfv_2]$, there are 6 of size 3,
\begin{center}
\begin{tabular}{cccccc}
\ $Q^g$ & $Q^h$ & $Q^i$ & $Q^l$ & $Q^n$ & $Q^o\ \ \ $\quad \\
$\Big[\SimS{220211}$, & $\SimS{422000}$, & $\SimS{332000}$, &
$\SimS{141110}$, & $\SimS{222110}$, & $\SimS{221111}\Big]_{S_3}$
\end{tabular}
\end{center}
and 7 of size 6,
\begin{center}
\begin{tabular}{ccccccc}
\ $Q^d$ & $Q^j$ & $Q^k$ & $Q^m$ & $Q^r$ & $Q^s$ & $Q^t\ \ \ \,$\\
$\Big[\SimS{410102}$, & $\SimS{412100}$, & $\SimS{322100}$, & $\SimS{132110}$, &
     $\SimS{131210}$, & $\SimS{221210}$, & $\SimS{121211}\Big]_{S_3}$.
\end{tabular}
\end{center}
For each of the above 8 choices of 21 simplex splines that are nontrivial on the boundary $\partial\PS$, we complete the basis by adding $S_3$ equivalence classes with together 18 elements, resulting in
\[ 8\left[{7 \choose 3} + {7\choose 2}{6\choose 2} + {7\choose 1}{6\choose 4} + {6 \choose 6}\right] = 3648 \]
potential $S_3$-invariant simplex spline bases $\BBB$ of $\SSS^3_5$ that reduce to a B-spline basis on the boundary. One selects the linearly independent sets $\BBB$ by checking that the collocation-like matrix $\{\lambda(Q)\}_{Q\in \BBB, \lambda \in\Lambda}$ has full rank.

\subsection{Positive partition of unity}
\noindent Let $\BBB = \{Q_1, \ldots, Q_{39}\}$ be an ordered simplex spline basis of $\SSS^3_5(\PS)$ that reduces to a B-spline basis on the boundary. We desire that $\BBB$ has a \emph{positive partition of unity}, i.e., that there exist weights $w_1, \ldots, w_{39} > 0$ for which
\begin{equation}\label{eq:PartitionOfUnity}
\sum_{i = 1}^{39} w_i Q_i(\bfx) = 1
\end{equation}
holds identically. Applying the functionals in $\Lambda = \{\lambda_1, \ldots, \lambda_{39}\}$ yields
\[ \sum_{i = 1}^{39} w_i \lambda_j(Q_i) = \lambda_j(1),\qquad j = 1,\ldots, 39, \]
which has a unique solution $(w_1,\ldots,w_{39})\in \RR^{39}$ by the linear independence of $\BBB$ and $\Lambda$. One then checks whether the weights $w_i$ are all positive.

\subsection{Marsden identity}
More generally, we would like $\BBB$ to satisfy a \emph{Marsden identity}
\begin{equation}\label{eq:MarsdenIdentity}
(1 + \bfx^\mT \bfy)^5 = \sum_{i = 1}^{39} w_i Q_i(\bfx) \psi_i(\bfy),\qquad \bfx\in \PS,\ \bfy \in \RR^2,
\end{equation}
for certain \emph{dual polynomials} $\psi_i$ and \emph{dual points} $\bfp_{i,r}^*$ of the form
\[ \psi_i (\bfy) := \prod_{r = 1}^5 (1 + \bfp_{i,r}^{*\mT} \bfy),
\qquad \bfp_{i,r}^*\in \RR^2, \quad r = 1, \ldots, 5, \quad i=1,\ldots,39. \]
In particular one recovers the partition of unity by setting $\bfy = 0$. Similarly one can generate all other polynomials of degree at most $5$. For instance, differentiating \eqref{eq:MarsdenIdentity} with respect to $\bfy$ and evaluating at $\bfy = 0$ yields
\begin{equation}\label{eq:DomainPoints}
\bfx = \sum_{i = 1}^{39} \bfxi_i w_i Q_i(\bfx),\qquad \bfxi_i := \left.\frac{1}{5} \nabla \psi_i\right|_{\bfy = 0} = \frac{\bfp_{i,1}^* + \cdots + \bfp_{i,5}^*}{5},
\end{equation}
where $\bfxi_i$ is the \emph{domain point} associated to $Q_i$.

While the Marsden identity \eqref{eq:MarsdenIdentity} is the form commonly encountered in the literature \cite{Goodman.Lee81, Hollig82}, we instead present a barycentric form that is independent of the vertices of the macrotriangle.

\begin{theorem}[Barycentric Marsden identity]
Let $\beta_j = \beta_j(\bfx), j = 1,2,3$, be the barycentric coordinates of $\bfx\in \RR^2$ with respect to $\PS = [\bfv_1,\bfv_2,\bfv_3]$. Then \eqref{eq:MarsdenIdentity} is equivalent to
\begin{equation}\label{eq:BarycentricMarsdenIdentity}
(\beta_1 c_1 + \beta_2 c_2 + \beta_3 c_3)^5 = \sum_{i = 1}^{39} w_i Q_i(\beta_1\bfv_1 + \beta_2\bfv_2 + \beta_3\bfv_3) \Psi_i(c_1,c_2,c_3),
\end{equation}
where $\bfx\in \PS$, $c_1,c_2,c_3\in \RR$, and, for $i=1,\ldots,39$,
\[ \Psi_i(c_1, c_2, c_3)
 := \prod_{r=1}^5 \big(\beta_1(\bfp_{i,r}^*)c_1 + \beta_2(\bfp_{i,r}^*)c_2 + \beta_3(\bfp_{i,r}^*)c_3\big). \]
\end{theorem}

\begin{proof}
Let $\bfX_j := (1,\bfv_j)^\mT\in \RR^3$, $j=1,2,3$. Since $\{\bfX_1, \bfX_2, \bfX_3\}$ is linearly independent, there are $\bfY_i\in \RR^3$, $i=1,2,3$, such that $\bfX_j^\mT\bfY_i = \delta_{ij}$. Given $\bfx\in \PS$ and $c_1,c_2,c_3\in \RR$ we define $\bfX, \bfY\in \RR^3$ by $\bfX := (1,\bfx)^\mT$ and $\bfY = (1, \bfy)^\mT := c_1\bfY_1 + c_2\bfY_2 + c_3\bfY_3$. Since $\bfX = \beta_1(\bfx)\bfX_1 + \beta_2(\bfx)\bfX_2 + \beta_3(\bfx)\bfX_3$,
\[ 1 + \bfx^\mT \bfy = \bfX^\mT\bfY = \beta_1(\bfx)c_1 + \beta_2(\bfx)c_2 + \beta_3(\bfx)c_3, \]
and in particular with $\bfx = \bfp_{i,r}^*$ and $\bfP_{i,r}^* := (1,\bfp_{i,r}^*)^\mT$, 
\[ 1 + \bfp_{i,r}^{*\mT} \bfy = \bfP_{i,r}^{*\mT}\bfY = \beta_1(\bfp_{i,r}^*)c_1 + \beta_2(\bfp_{i,r}^*)c_2 + \beta_3(\bfp_{i,r}^*)c_3. \]
It follows that \eqref{eq:BarycentricMarsdenIdentity} is equivalent to \eqref{eq:MarsdenIdentity}.
\end{proof}

For a compact representation of the dual polynomials, we introduce the shorthands (compare Figure \ref{fig:PS12-Labels1})
\begin{equation}\label{eq:ShorthandDual}
\begin{aligned}
& c_4 := \frac{c_1 + c_2}{2},\quad c_5 := \frac{c_2 + c_3}{2},\quad c_6 := \frac{c_1 + c_3}{2},\\
& c_7 := \frac{c_4 + c_6}{2},\quad c_8 := \frac{c_4 + c_5}{2},\quad c_9 := \frac{c_5 + c_6}{2},\quad
c_{10} := \frac{c_1 + c_2 + c_3}{3}.
\end{aligned}
\end{equation}

\begin{example}
The barycentric Marsden identity for the quadratic S-basis in \cite[Theorem 3.1]{Cohen.Lyche.Riesenfeld13} (cf. Example \ref{ex:S-basis}) is
\begin{align*}
& (c_1\beta_1 + c_2\beta_2 + c_3\beta_3)^2 = \left(c_1 \SimS{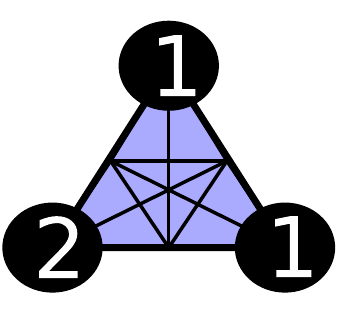} + c_2\SimS{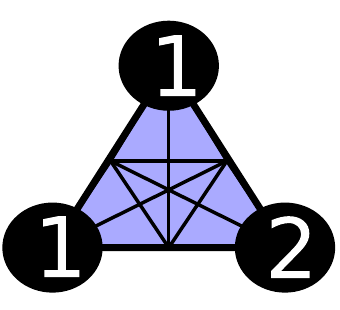} + c_3\SimS{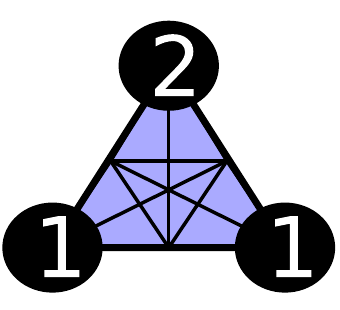}\right)^2 = \\
& \qquad + \frac14 c_1^2 \SimS{300101} + \frac34 c_4c_{10}\SimS{110111} + \frac12 c_1c_4 \SimS{210101}+ \frac12 c_2c_4 \SimS{120110}\\
& \qquad + \frac14 c_2^2 \SimS{030110} + \frac34 c_5c_{10}\SimS{011111} + \frac12 c_2c_5 \SimS{021110} + \frac12 c_3c_5 \SimS{012011}\\
& \qquad + \frac14 c_3^2 \SimS{003011} + \frac34 c_6c_{10}\SimS{101111} + \frac12 c_3c_6 \SimS{102011} + \frac12 c_1c_6 \SimS{201101},
\end{align*}
where the weigths follow from \eqref{eq:S-splines}.
\end{example}

For every basis $\BBB = \{Q_1,\ldots, Q_{39}\}$ with a positive partition of unity \eqref{eq:PartitionOfUnity}, we apply the functionals in $\Lambda$ to \eqref{eq:MarsdenIdentity} with respect to $\bfx$, which gives
\begin{equation*}
\begin{pmatrix}
\lambda_1   (Q_1) & \cdots & \lambda_1   (Q_{39})\\
\vdots            & \ddots & \vdots\\
\lambda_{39}(Q_1) & \cdots & \lambda_{39}(Q_{39})
\end{pmatrix}
\begin{pmatrix} w_1\psi_1(\bfy)\\ \vdots\\ w_{39}\psi_{39}(\bfy) \end{pmatrix}
=
\begin{pmatrix}
\lambda_1\big((1 + \bfx^\mT \bfy)^5\big)\\
\vdots\\
\lambda_{39}\big((1 + \bfx^\mT \bfy)^5\big)
\end{pmatrix}.
\end{equation*}
This system has a unique solution in the module $\RR[\bfy]^{39}$, with each component $w_i \psi_i(\bfy)$ a polynomial of degree at most 5 by Cramer's rule. The basis $\BBB$ has a Marsden identity \eqref{eq:MarsdenIdentity} if and only if these polynomials split into real linear factors.

\section{Main results}\label{sec:Results}
\noindent Some of the remaining computations are too large to carry out by hand. We have therefore implemented the above computations in {\tt Sage}~\cite{Sage}. From the website \cite{WebsiteGeorg} of the second author, the resulting worksheet can be downloaded and tried out online in {\tt SageMathCloud}.

\subsection{Six bases for $\SSS^3_5$}
Checking linear independence for each of the 3648 potential bases, we discover that there are 1024 $S_3$-invariant simplex spline bases that reduce to a B-spline basis on the boundary. There are 243 such bases with a nonnegative partition of unity, of which there are 47 bases with a positive partition of unity. Of these there are 9 with all domain points inside the macrotriangle, of which there are 7 with precisely 8 domain points on each boundary edge. Of these there are 6 bases $\BBB_a, \BBB_b, \BBB_c, \BBB_d, \BBB_e, \BBB_f$ for which each dual polynomial only has real linear factors. These bases, together with their weights $w_\star$, dual polynomials $\Psi_\star$, and domain points $\bfxi_\star$, are listed in Table \ref{tab:DualPolynomials}. For instance, the highlighted rows in the table yield the set
\begin{equation}
\left[ \frac14 \SimS{600101}, \frac14 \SimS{500201}, \frac12 \SimS{410201}, \frac12 \SimS{320201}, \frac34 \SimS{220211}, \SimS{141110}, \frac12 \SimS{131210}, \frac34 \SimS{121211} \right]_{S_3}. \tag{$\BBB_c$}
\end{equation}
We summarize these results in the following theorem.

\begin{table}
\begin{adjustwidth}{-1.5in}{-1.5in}
\begin{center}
\begin{tabular}{cccccccccccccc}
\toprule
 & & $Q^a$ & $Q^b$ & $Q^e$ & $Q^f$ & $Q^g$ & $Q^l$\\
 & & \SimS{600101} & \SimS{500201} & \SimS{410201} & \SimS{320201} & \SimS{220211} & \SimS{141110} \\ \midrule
$\BBB_a$ & $w_\star\Psi_\star$ & $\frac{1}{4}c_1^5$ & $\frac{1}{4}c_1^4c_4$ & $\frac{1}{2}c_1^3c_4^2$ & $\frac{1}{2}c_1^2c_2c_4^2$ & 0 & $c_2^3c_4c_5$ \\
 & $\bfxi_\star$ & $\left(1,0,0\right)$ & $\left(\frac{9}{10}, \frac{1}{10}, 0\right)$ & $\left(\frac45, \frac15, 0\right)$ & $\left(\frac35, \frac25, 0\right)$ & $\times$ & $\left(\frac{1}{10},\frac45, \frac{1}{10}\right)$ \\  \midrule
$\BBB_b$ & $w_\star\Psi_\star$ & $\frac{1}{4}c_1^5$ & $\frac{1}{4}c_1^4c_4$ & $\frac{1}{2}c_1^3c_4^2$ & $\frac{1}{2}c_1^2c_2c_4^2$ & 0 & $c_2^3c_4c_5$ \\
 & $\bfxi_\star$ & $\left(1,0,0\right)$ & $\left(\frac{9}{10}, \frac{1}{10}, 0\right)$ & $\left(\frac45, \frac15, 0\right)$ & $\left(\frac35, \frac25, 0\right)$ & $\times$ & $\left(\frac{1}{10},\frac45, \frac{1}{10}\right)$ \\  \midrule
\rowcolor{LightGray} $\BBB_c$ & $w_\star\Psi_\star$ & $\frac{1}{4}c_1^5$ & $\frac{1}{4}c_1^4c_4$ & $\frac{1}{2}c_1^3c_4^2$ & $\frac{1}{2}c_1^2c_2c_4^2$ & $\frac{3}{4}c_1c_2c_4^2c_{10}$ & $c_2^3c_4c_5$ \\
\rowcolor{LightGray} & $\bfxi_\star$ & $\left(1,0,0\right)$ & $\left(\frac{9}{10}, \frac{1}{10}, 0\right)$ & $\left(\frac45, \frac15, 0\right)$& $\left(\frac35, \frac25, 0\right)$ & $\left(\frac{7}{15}, \frac{7}{15}, \frac{1}{15}\right)$ & $\left(\frac{1}{10},\frac45, \frac{1}{10}\right)$ \\ \midrule
$\BBB_d$ & $w_\star\Psi_\star$ & $\frac{1}{4}c_1^5$ & $\frac{1}{4}c_1^4c_4$ & $\frac{1}{2}c_1^3c_4^2$ & 0 & 0 & $c_2^3c_4c_5$ \\
 & $\bfxi_\star$ & $\left(1,0,0\right)$ & $\left(\frac{9}{10}, \frac{1}{10}, 0\right)$ & $\left(\frac45, \frac15, 0\right)$ & $\times$ & $\times$ & $\left(\frac{1}{10},\frac45, \frac{1}{10}\right)$ \\  \midrule
$\BBB_e$ & $w_\star\Psi_\star$ & $\frac{1}{4}c_1^5$ & $\frac{1}{4}c_1^4c_4$ & $\frac{1}{2}c_1^3c_4^2$ & 0 & 0 & $c_2^3c_4c_5$ \\
 & $\bfxi_\star$ & $\left(1,0,0\right)$ & $\left(\frac{9}{10}, \frac{1}{10}, 0\right)$ & $\left(\frac45, \frac15, 0\right)$ & $\times$ & $\times$ & $\left(\frac{1}{10},\frac45, \frac{1}{10}\right)$ \\  \midrule
$\BBB_f$ & $w_\star\Psi_\star$ & $\frac{1}{4}c_1^5$ & $\frac{1}{4}c_1^4c_4$ & $\frac{1}{2}c_1^3c_4^2$ & 0 & $\frac{1}{4}c_1c_2c_3c_4^2$ & $c_2^3c_4c_5$ \\
 & $\bfxi_\star$ & $\left(1,0,0\right)$ & $\left(\frac{9}{10}, \frac{1}{10}, 0\right)$ & $\left(\frac45, \frac15, 0\right)$ & $\times$ & $\left(\frac25,\frac25,\frac15\right)$ & $\left(\frac{1}{10},\frac45, \frac{1}{10}\right)$ \\ \bottomrule
 & & $Q^n$ & $Q^o$ & $Q^q$ & $Q^r$ & $Q^s$ & $Q^t$\\
 & & \SimS{222110} & \SimS{221111} & \SimS{321200} & \SimS{131210} & \SimS{221210} & \SimS{121211} \\ \midrule
$\BBB_a$ & $w_\star\Psi_\star$ & $\frac{3}{4}c_1c_2^2c_3c_{10}$ & 0 & 0 & $\frac{1}{2}c_1c_2^2c_4c_5$ & $\frac{3}{4}c_1c_2^2c_4c_{10}$ & 0 \\
 & $\bfxi_\star$ & $\left(\frac{4}{15}, \frac{7}{15}, \frac{4}{15}\right)$ & $\times$ & $\times$ & $\left(\frac{3}{10}, \frac{3}{5}, \frac{1}{10}\right)$ & $\left(\frac{11}{30}, \frac{17}{30}, \frac{1}{15}\right)$ & $\times$ \\  \midrule
$\BBB_b$ & $w_\star\Psi_\star$ & 0 & $\frac{3}{4}c_1c_2c_3c_4c_{10}$ & 0 & $\frac{1}{2}c_1c_2^2c_4c_5$ & $\frac{3}{4}c_1c_2^2c_4c_{10}$ & 0 \\
 & $\bfxi_\star$ & $\times$ & $\left(\frac{11}{30}, \frac{11}{30}, \frac{4}{15}\right)$ & $\times$ & $\left(\frac{3}{10}, \frac{3}{5}, \frac{1}{10}\right)$ & $\left(\frac{11}{30}, \frac{17}{30}, \frac{1}{15}\right)$ & $\times$  \\  \midrule
\rowcolor{LightGray} $\BBB_c$ & $w_\star\Psi_\star$ & 0 & 0 & 0 & $\frac{1}{2}c_1c_2^2c_4c_5$ & 0 & $\frac{3}{4}c_1c_2c_4c_5c_{10}$ \\
\rowcolor{LightGray} & $\bfxi_\star$ & $\times$ & $\times$ & $\times$ & $\left(\frac{3}{10}, \frac{3}{5}, \frac{1}{10}\right)$ & $\times$ & $\left(\frac{11}{30}, \frac{7}{15}, \frac{1}{6}\right)$ \\ \midrule
$\BBB_d$ & $w_\star\Psi_\star$ & $\frac{3}{4}c_1c_2^2c_3c_{10}$ & 0 & $c_1^2c_2c_4^2$ & $\frac{1}{2}c_1c_2^2c_4c_5$ & $\frac{1}{4}c_1c_2^2c_3c_4$ & 0 \\
 & $\bfxi_\star$ & $\left(\frac{4}{15}, \frac{7}{15}, \frac{4}{15}\right)$ & $\times$ & $\left(\frac{3}{5}, \frac{2}{5}, 0\right)$ & $\left(\frac{3}{10}, \frac{3}{5}, \frac{1}{10}\right)$ & $\left(\frac{3}{10}, \frac{1}{2}, \frac{1}{5}\right)$ & $\times$ \\  \midrule
$\BBB_e$ & $w_\star\Psi_\star$ & 0 & $\frac{3}{4}c_1c_2c_3c_4c_{10}$ & $c_1^2c_2c_4^2$ & $\frac{1}{2}c_1c_2^2c_4c_5$ & $\frac{1}{4}c_1c_2^2c_3c_4$ & 0 \\
 & $\bfxi_\star$ & $\times$ & $\left(\frac{11}{30}, \frac{11}{30}, \frac{4}{15}\right)$ & $\left(\frac{3}{5}, \frac{2}{5}, 0\right)$ & $\left(\frac{3}{10}, \frac{3}{5}, \frac{1}{10}\right)$ & $\left(\frac{3}{10}, \frac{1}{2}, \frac{1}{5}\right)$ & $\times$ \\  \midrule
$\BBB_f$ & $w_\star\Psi_\star$ & 0 & 0 & $c_1^2c_2c_4^2$ & $\frac{1}{2}c_1c_2^2c_4c_5$ & 0 & $\frac{1}{2}c_1c_2c_3c_4c_8$ \\
 & $\bfxi_\star$ & $\times$ & $\times$ & $\left(\frac{3}{5}, \frac{2}{5}, 0\right)$ & $\left(\frac{3}{10}, \frac{3}{5}, \frac{1}{10}\right)$ & $\times$ & $\left(\frac{7}{20}, \frac{2}{5}, \frac{1}{4}\right)$ \\
 \bottomrule
\end{tabular}
\end{center}
\end{adjustwidth}
\bigskip
\caption{For the six bases $\BBB_a,\ldots,\BBB_f$ and for each contained representative $Q^\star$ in Table \ref{tab:AllSimplexSplines532}, the table lists the weight $w_\star$ times the dual polynomial $\Psi_\star$, using the shorthands in~\eqref{eq:ShorthandDual}, and the barycentric coordinates of the domain point $\bfxi_\star$.}\label{tab:DualPolynomials}
\end{table}

\begin{figure}
\begin{center}
\subfloat[]{\includegraphics[scale=0.66]{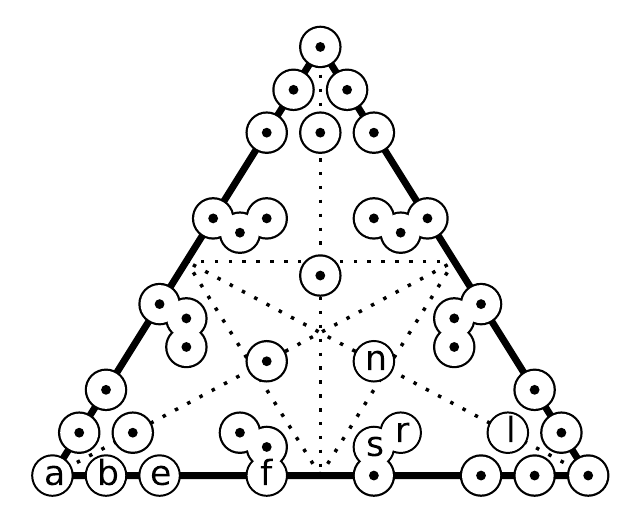}\label{fig:Marsden317}}
\subfloat[]{\includegraphics[scale=0.66]{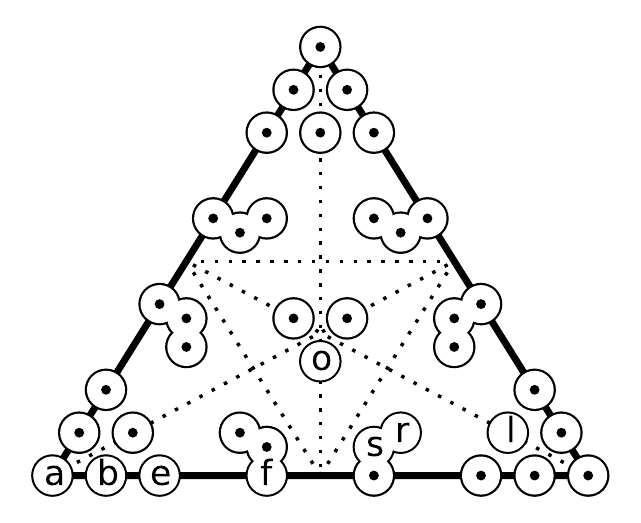}\label{fig:Marsden318}}
\subfloat[]{\includegraphics[scale=0.66]{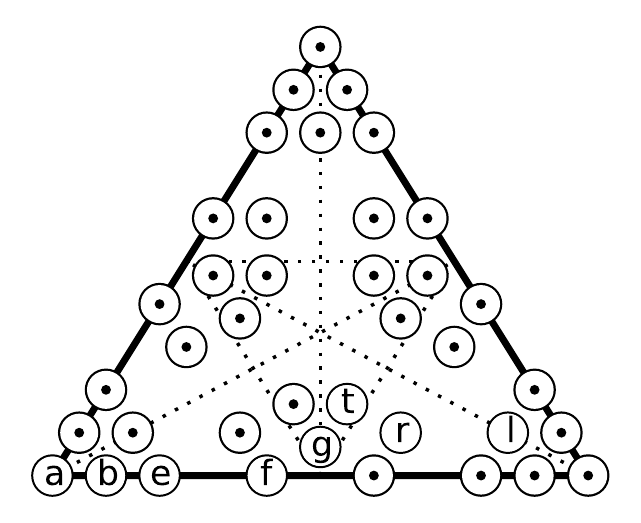}\label{fig:Marsden322}}\\
\subfloat[]{\includegraphics[scale=0.66]{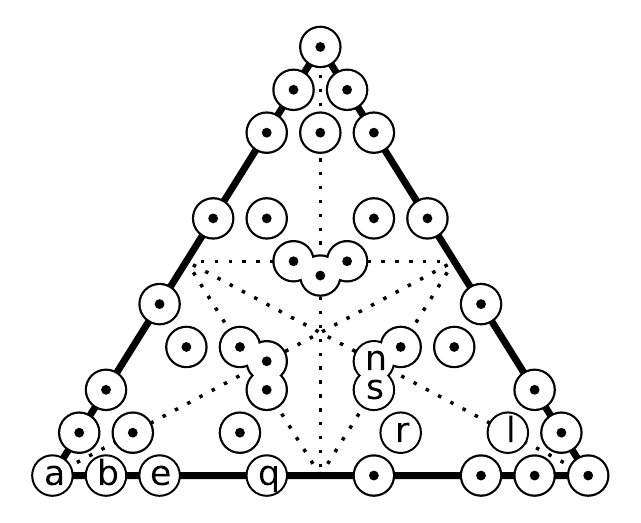}\label{fig:Marsden773}}
\subfloat[]{\includegraphics[scale=0.66]{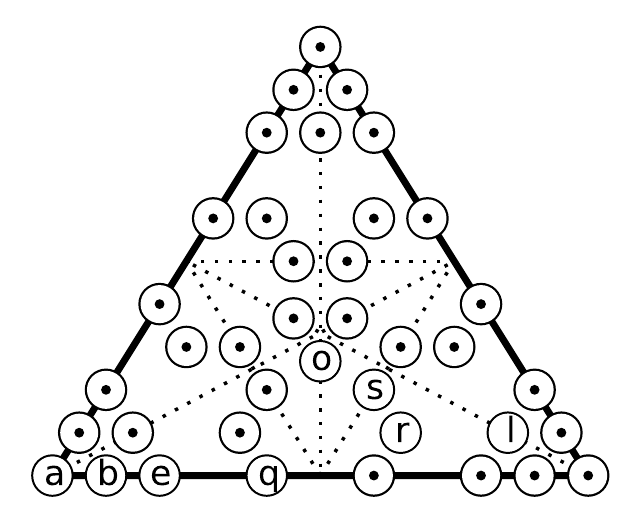}\label{fig:Marsden774}}
\subfloat[]{\includegraphics[scale=0.66]{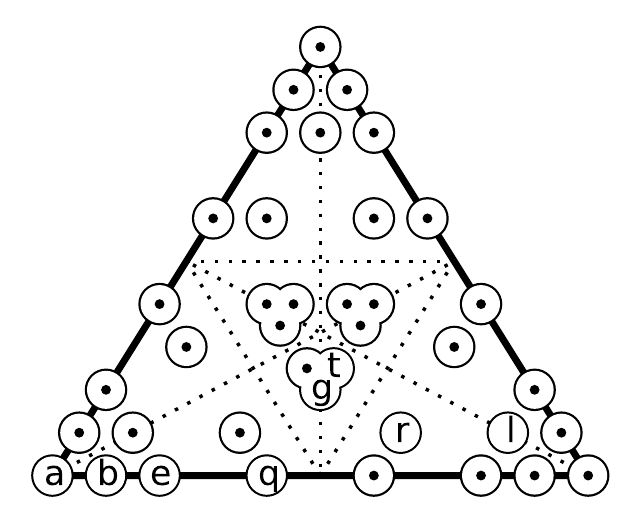}\label{fig:Marsden778}}\\
\end{center}
\caption{For the six bases in Table~\ref{tab:DualPolynomials}, the figure shows the domain points $\odot$ and, for each $S_3$ equivalence class, the label of a representative in Table \ref{tab:AllSimplexSplines532}. }\label{fig:MarsdenDomain}
\end{figure}

\begin{theorem}
The sets $\BBB = \BBB_a, \BBB_b, \BBB_c, \BBB_d, \BBB_e, \BBB_f$ are the only sets satisfying:
\begin{enumerate}
\item $\BBB$ is a basis of $\SSS^3_5(\PSB)$ consisting of simplex splines.
\item $\BBB$ is $S_3$-invariant.
\item $\BBB$ reduces to a B-spline basis on the boundary.
\item $\BBB$ has a positive partition of unity and a Marsden identity \eqref{eq:BarycentricMarsdenIdentity}, for which the dual polynomials have only real linear factors; see Table~\ref{tab:DualPolynomials}.
\item $\BBB$ has all its domain points inside the macrotriangle $\PS$, with precisely 8 domain points on each edge of $\PS$.
\end{enumerate}
\end{theorem}

\begin{remark}
Table \ref{tab:DualPolynomials} shows that, for the simplex splines
\[ \left[Q^g, Q^s, Q^t\right]_{S_3} = \left[ \SimS{220211}, \SimS{221210}, \SimS{121211} \right]_{S_3}, \]
the weights and dual polynomials depend on the entire basis, and they cannot be determined directly from the corresponding knot multiset.
\end{remark}

\subsubsection*{Dual points and domain points}
One immediately reads off the dual points from the dual polynomials in Table \ref{tab:DualPolynomials}, simply by replacing `$c$' by `$\bfv$'. For instance, in each basis the simplex spline
\[ Q^l = \SimS{141110} \]
has dual polynomial $\Psi^l = c_2^3 c_4 c_5$, and therefore dual points $\{\bfv_2^3\,\bfv_4\,\bfv_5\}$. By \eqref{eq:DomainPoints} the corresponding domain point is $(3\bfv_2 + \bfv_4 + \bfv_5)/5 = \frac{1}{10}\bfv_1 + \frac{4}{5}\bfv_2 + \frac{1}{10}\bfv_3$. Thus one obtains for each basis all domain points, which are listed in Table~\ref{tab:DualPolynomials} and shown in Figure \ref{fig:MarsdenDomain}. The domain points of the basis $\BBB_c$ are connected to form the domain mesh in Figure \ref{fig:domain_mesh_Bc}. To preserve the symmetry of $\PSB$, the domain points are forced to form a hybrid mesh with triangles, quadrilaterals, and a hexagon in the center.

\subsubsection*{Polynomial reproduction}
Following \cite{Knuth94}, we define for any nonnegative integers $i_1,i_2,i_3$ the ``coefficient of'' operator
\[ \left[c_1^{i_1} c_2^{i_2} c_3^{i_3} \right]f = \frac{1}{i_1!i_2!i_3!} \frac{\partial^5 f}{\partial c_1^{i_1}\partial c_2^{i_2}\partial c_3^{i_3}}(0, 0, 0) \]
for any formal power series $f(c_1, c_2, c_3)$. Substituting the dual polynomials from Table \ref{tab:DualPolynomials} and the shorthands from \eqref{eq:ShorthandDual} into \eqref{eq:BarycentricMarsdenIdentity} and applying $\left[c_1^{i_1} c_2^{i_2} c_3^{i_3} \right]$, with $i_1 + i_2 + i_3 = 5$, we recover the Bernstein polynomials
\[
\frac{5!}{i_1!i_2!i_3!} \beta_1^{i_1} \beta_2^{i_2} \beta_3^{i_3} =
\sum_{i = 1}^{39} w_i Q_i(\beta_1\bfv_1 + \beta_2\bfv_2 + \beta_3\bfv_3)
\left[c_1^{i_1} c_2^{i_2} c_3^{i_3} \right] \Psi_i. 
\]
Thus one sees immediately from the monomials in the dual polynomials which simplex splines appear in the above linear combination. For example, the Bernstein polynomial $\beta_1^5$ corresponds to the lattice vector $(i_1,i_2,i_3) = (5,0,0)$, and
\[ \SimS{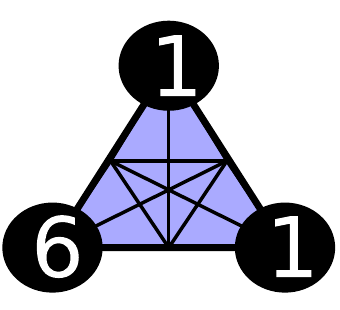} = \frac14 \SimS{600101} + \frac18 \SimS{500201} + \frac18 \SimS{500102} + \frac18 \SimS{410201} + \frac18 \SimS{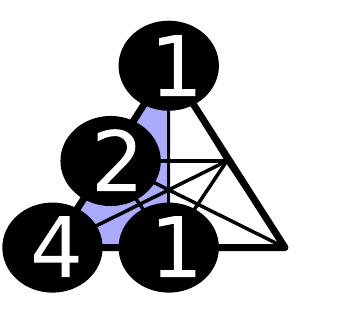} + \frac14 \SimS{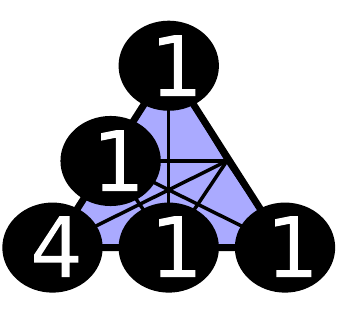}. \]

\subsubsection*{Quasi-interpolation}
For each basis $\BBB=\{S_1,\ldots,S_{39}\}$ with dual points $\bfp^*_{i,r}$, consider the map $Q: C^0(\PS)\rightarrow \SSS^3_5(\PSB)$ defined by $Q(f) = \sum_{i=1}^{39} L_i(f) S_i$ and
\[ L_i(f) = \sum_{k=1}^5 \frac{k^5}{5!}  (-1)^{k-1} \sum_{1\le r_1 < \cdots < r_k \le 5} f\left(\frac{\boldsymbol{p}^*_{i,r_1} + \cdots + \boldsymbol{p}^*_{i,r_k}}{k}\right). \]
Note that this is an affine combination of function values of $f$, i.e.,
\[ \sum_{k=1}^5 \frac{k^5}{5!} (-1)^{k-1} {5\choose k} = 1,\]
implying that $Q$ reproduces constants. Moreover, using the Marsden identity it is easily checked \cite{WebsiteGeorg} that $Q$ reproduces polynomials up to degree 5, i.e., $Q(B^5_{ijk}) = B^5_{ijk}$, whenever $i+j+k=5$. Also, $Q$ is bounded independently of the geometry of $\PS$, since, using that $\BBB$ forms a partition of unity,
\[ \|Q(f)\|_{L_\infty(\PSsmall)} 
\leq \max_i |L_i(f)|
\leq \frac{275}{3} \|f\|_{L_\infty(\PSsmall)}.
\]
Therefore, by a standard argument, $Q$ is a quasi-interpolant that approximates locally with order 6 smooth functions whose first six derivatives are in $L_\infty(\PS)$. Note that $Q$ does not reproduce all splines in $\SSS^3_5(\PSB)$.

\subsubsection*{$L_\infty$ stability and distance to the control points}
The next theorem shows that each basis is stable in the $L_\infty$ norm with a condition number bounded independent of the geometry of $\PS$.

\begin{theorem}
Let $\BBB_\star = \{S_1,\ldots,S_{39}\}$ be one of the bases $\BBB_a, \ldots, \BBB_f$, and  $f = \bfS^T \bfc \in \SSS^3_5(\PSB)$ with $\bfS^T =(S_1,\ldots,S_{39})$ and $\bfc = (c_1,\ldots,c_{39})^T$. Then there is a constant $K_\star > 0$ independent of the geometry of $\PS$, such that
\begin{equation}\label{eq:stablebases}
K^{-1}_\star \|\bfc\|_\infty 
\le \|f\|_{L_\infty(\PSsmall)} \leq \|\bfc\|_\infty.
\end{equation}
\end{theorem}

\begin{proof}
Let $\bfxi_1, \ldots,\bfxi_{39}$ be the domain points of $\BBB_\star$. A calculation shows that the collocation matrix $M_\star = (m_{ij})_{i,j=1}^{39}$, with $m_{ij} = S_j(\bfxi_i)$ is nonsingular, and its elements are rational numbers independent of the geometry of $\PS$. 

Using the Lagrange interpolant, the coefficients take the form $\bfc = M_\star^{-1} \bff$, where $\bff = (f(\bfxi_1),\ldots, f(\bfxi_{39}))^T$. Hence, since $\bfS$ forms a partition of unity, \eqref{eq:stablebases} holds with $K_\star = \|M^{-1}_\star\|_\infty$. 
\end{proof}

Note that $K_\star$ is an upper bound for the condition number of the basis $\BBB_\star$, and is in fact the infinity norm condition number of the matrix $M_\star$, because $\|M_\star\|_\infty = 1$. The smallest constant $\|M^{-1}_\star\|_\infty$ is obtained for $\BBB_c$, in which case
\[ K_c = \|M_c^{-1}\|_\infty = \frac{60866923187443943219194678615331}{836197581250152380489105335680} \approx 72.7901.\]
Hence, there is a well-conditioned Lagrange interpolant at the domain points of the basis $\BBB_c$.

We can now bound the distance between the B\'ezier ordinates and the values of a spline at the corresponding domain points.
\begin{corollary}
Let $h$ be the longest edge in $\PS$, and let $f = \bfS^T \bfc$ with Hessian matrix $H$ and values $\bff = (f(\bfxi_1),\ldots, f(\bfxi_{39}))^T$. Then 
\[ \|\bfc - \bff\|_\infty \le 2 K_\star h^2 \max_{\bfx \in \PSsmall} \|H(\bfx)\|_\infty. \]
\end{corollary}

\begin{proof}
Consider the first-order Taylor expansion of $f$ at $\bfx_0$,
\[ f(\bfx) - f(\bfx_0) = (\bfx-\bfx_0) \nabla f(\bfx_0) + g(\bfx). \]
As the error term $g \in \SSS^3_5(\PSB)$, it takes the form $g = \sum_{i=1}^{39} b_i S_i$, with
\[ c_i - f(\bfx_0) = (\bfxi_i - \bfx_0)\nabla f(\bfx_0) + b_i. \]
In particular for $\bfx_0 = \bfxi_i$, we obtain
\begin{align*} |c_i - f(\bfxi_i)| = |b_i|
& \le K_\star \max_{\bfx\in \PSsmall} |g(\bfx)|\\
& \le \frac{1}{2} K_\star \max_{\bfx, \bfy\in \PSsmall} \left|(\bfx - \bfxi_i)^T H(\bfy) (\bfx - \bfxi_i)\right|,
\end{align*}
from which the theorem follows.
\end{proof}

\subsection{The basis $\BBB_c$}
While the remainder of the paper can be carried out for all six bases, we now restrict our discussion to the basis $\BBB_c = \{w_i Q_i\}_{i=1}^{39}$ for several reasons. First of all, the condition number $K_c$ is smallest. Secondly, this basis has the most localized support, because it contains the splines
\[ \left[\frac12 \SimS{320201}, \frac34 \SimS{220211}\right]_{S_3}, \]
as opposed to splines with full support. Finally, for $k=0,1,2$ and any direction $\bfu$ not parallel to the edge $e$ of $\PS$, the number of additional splines $Q_i$ for which $D^k_\bfu Q_i |_e$ is nonzero corresponds to the dimension of the space of univariate splines of degree $5 - k$ on the knot multiset $\{0^{d+1-k}\,0.5^2\, 1^{d+1-k}\}$. This allows for relatively pretty smoothness conditions analogous to the Bernstein-B\'ezier case.

\begin{table}[t]
{\small
\begin{tabular*}{\columnwidth}{@{ }@{\extracolsep{\stretch{1}}}*{4}{l}@{ }}
\toprule
$B_1^5 := B[0^6\,0.5^1]$     & $B_1^4 := B[0^5\,0.5^1]$     & $B_1^3 := B[0^4\,0.5^1]$     & $B_1^2 := B[0^3\,0.5^1]$\\
$B_2^5 := B[0^5\,0.5^2]$     & $B_2^4 := B[0^4\,0.5^2]$     & $B_2^3 := B[0^3\,0.5^2]$     & $B_2^2 := B[0^2\,0.5^2]$\\
$B_3^5 := B[0^4\,0.5^2\,1^1]$& $B_3^4 := B[0^3\,0.5^2\,1^1]$& $B_3^3 := B[0^2\,0.5^2\,1^1]$& $B_3^2 := B[0^1\,0.5^2\,1^1]$\\
$B_4^5 := B[0^3\,0.5^2\,1^2]$& $B_4^4 := B[0^2\,0.5^2\,1^2]$& $B_4^3 := B[0^1\,0.5^2\,1^2]$& $B_4^2 := B[0.5^2\,1^2]$\\
$B_5^5 := B[0^2\,0.5^2\,1^3]$& $B_5^4 := B[0^1\,0.5^2\,1^3]$& $B_5^3 := B[0.5^2\,1^3]$     & $B_5^2 := B[0.5^1\,1^3]$\\
$B_6^5 := B[0^1\,0.5^2\,1^4]$& $B_6^4 := B[0.5^2\,1^4]$     & $B_6^3 := B[0.5^1\,1^4]$     &\\
$B_7^5 := B[0.5^2\,1^5]$     & $B_7^4 := B[0.5^1\,1^5]$     & \\
$B_8^5 := B[0.5^1\,1^6]$     &  & \\
\bottomrule
\end{tabular*}
}
\caption{Shorthands for univariate B-splines of degree $d$ on the open knot multiset $\{0^{d+1}\,0.5^2\,1^{d+1}\}$, for $d = 5,4,3,2$.}\label{tab:shorthandsBi}
\end{table}

\begin{figure}
\begin{center}
\includegraphics[scale=1.05, bb = 0 10 345 200]{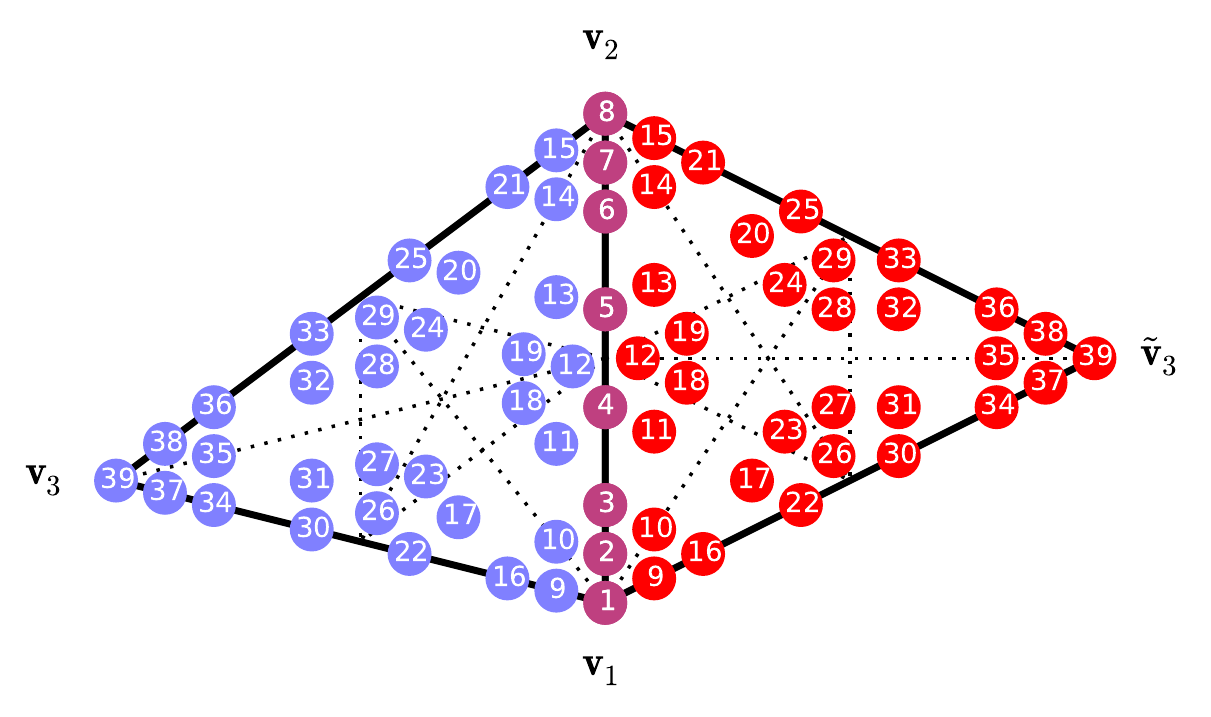}
\end{center}
\caption[]{The domain points of the bases $\BBB_c$ and $\tilde{\BBB}_c$, ordered by a decreasing number of knots on $[\bfv_1, \bfv_2]$, i.e., 1 knot for $Q_1, \ldots, Q_8$, 2 knots for $Q_9, \ldots, Q_{15}$, 3 knots for $Q_{16}, \ldots, Q_{21}$, 4 knots for $Q_{22}, \ldots, Q_{25}$, and more than $4$ knots for $Q_{26},\ldots, Q_{39}$. }\label{fig:basis-order}
\end{figure}

\begin{table}
\thisfloatpagestyle{empty}
\begin{adjustwidth}{-0.15\columnwidth}{-0.15\columnwidth}
\vspace{-2.5em}
\centering\begin{tabular}{ccccc}
\toprule
$i$ & $Q_i$ & $Q_i|_e$
           & $\frac{1}{5} D_{(\alpha_1, \alpha_2, \alpha_3)} Q_i|_e$
           & $\frac{1}{4\cdot 5} D^2_{(\alpha_1, \alpha_2, \alpha_3)} Q_i|_e$ \\ \midrule
1 & \SimS{600101} & $ 4 B_1^5$
              & $ 8\alpha_1 B_1^4$
              & $ 16\alpha_1^2 B_1^3$\\
2 & \SimS{500201} & $ 4 B_2^5$
              & $ 8(\alpha_2 B_1^4 + \alpha_1 B_2^4)$
              & $ 32\alpha_1\alpha_2 B_1^3 + 16 \alpha_1^2 B_2^3$\\
3 & \SimS{410201} & $ 2 B_3^5$
              & $ 4\alpha_2 B_2^4 + 2(2\alpha_1 + \alpha_2) B_3^4$
              & $ 8\alpha_2^2 B_1^3 + 4\alpha_2(4\alpha_1 + \alpha_2) B_2^3 + 2(2\alpha_1 + \alpha_2)^2 B_3^3$\\
4 & \SimS{320201} & $ 2 B_4^5$
              & $ 2\alpha_2 B_3^4 + 2(2\alpha_1 + \alpha_2) B_4^4$
              & $ 4\alpha_2^2 B_2^3 + 4\alpha_2(2\alpha_1 + \alpha_2) B_3^3 + 2(2\alpha_1 + \alpha_2)^2 B_4^3$\\
5 & \SimS{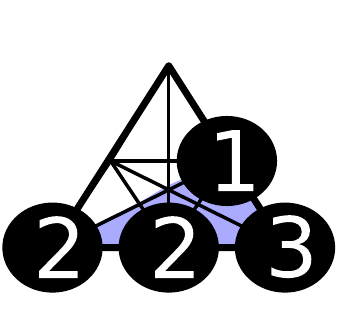} & $ 2 B_5^5$
              & $ 2(\alpha_1 + 2\alpha_2) B_4^4 + 2\alpha_1 B_5^4$
              & $ 2(\alpha_1 + 2\alpha_2)^2 B_3^3 + 4\alpha_1(\alpha_1 + 2\alpha_2)B_4^3 + 4\alpha_1^2 B_5^3$\\
6 & \SimS{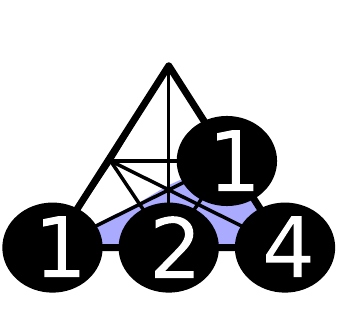} & $ 2 B_6^5$
              & $ 2(\alpha_1 + 2\alpha_2) B_5^4 + 4\alpha_1 B_6^4$
              & $ 2(\alpha_1 + 2\alpha_2)^2 B_4^3 + 4\alpha_1(\alpha_1 + 4\alpha_2)B_5^3 + 8\alpha_1^2 B_6^3$\\
7 & \SimS{050210} & $ 4 B_7^5$
              & $ 8(\alpha_2B_6^4 + \alpha_1 B_7^4)$
              & $ 16\alpha_2^2 B_5^3 + 32\alpha_1\alpha_2 B_6^3$\\
8 & \SimS{060110} & $ 4 B_8^5$
              & $ 8\alpha_2 B_7^4$
              & $ 16\alpha_2^2 B_6^3$\\ \midrule
9 & \SimS{500102} & 0 & $8\alpha_3 B_1^4$
                  & $32\alpha_1\alpha_3 B_1^3$\\
10 & \SimS{411101} & 0 & $ \alpha_3(2B_2^4 + B_3^4)$
      & $8\alpha_2\alpha_3 B_1^3 + \alpha_3(3\alpha_1 + \alpha_2) (2 B_2^3 + B_3^3)$\\
11 & \SimS{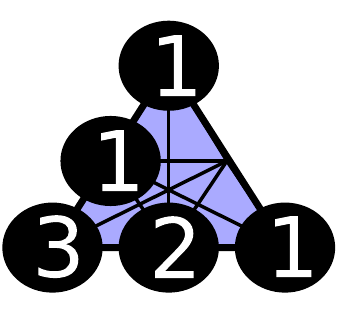} & 0 & $2\alpha_3 B_3^4$
                  & $8\alpha_2\alpha_3 B_2^3 + 2\alpha_3(3\alpha_1 + \alpha_2) B_3^3$\\
12 & \SimS{220211} & 0 & $4\alpha_3 B_4^4$
                  & $4\alpha_3(\alpha_1 + 3\alpha_2)B_3^3 + 4\alpha_3(3\alpha_1 + \alpha_2) B_4^3$\\
13 & \SimS{131210} & 0 & $2\alpha_3 B_5^4$
                  & $2\alpha_3(\alpha_1 + 3\alpha_2)B_4^3 + 8\alpha_1\alpha_3 B_5^3$\\
14 & \SimS{141110} & 0 & $ \alpha_3(B_5^4 + 2B_6^4)$
      & $\alpha_3(\alpha_1 + 3\alpha_2) (B_4^3 + 2 B_5^3) + 8 \alpha_1\alpha_3 B_6^3$\\
15 & \SimS{050120} & 0 & $8\alpha_3 B_7^4$
                  & $32\alpha_2\alpha_3 B_6^3$\\ \midrule
16 & \SimS{401102} & 0 & 0 & $8\alpha_3^2 B_1^3$\\
17 & \SimS{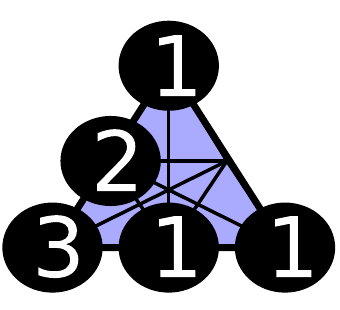} & 0 & 0 & $2\alpha_3^2(2B_2^3 + B_3^2)$\\
18 & \SimS{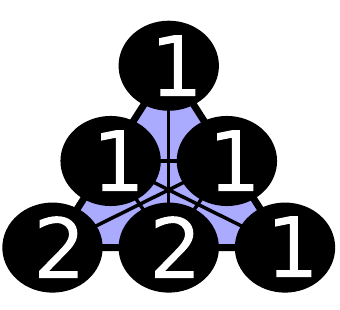} & 0 & 0 & $4\alpha_3^2 B_3^2$\\
19 & \SimS{121211} & 0 & 0 & $4\alpha_3^2 B_4^3$\\
20 & \SimS{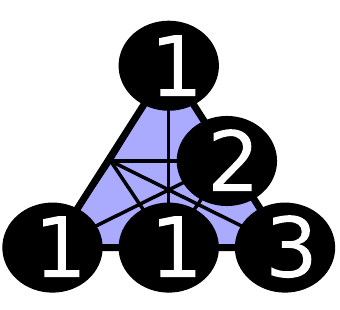} & 0 & 0 & $2\alpha_3^2(B_4^3 + 2B_5^3)$\\
21 & \SimS{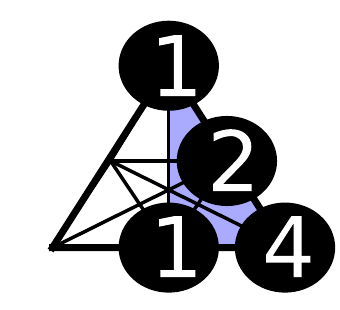} & 0 & 0 & $8\alpha_3^2 B_6^3$\\ \midrule
22 & \SimS{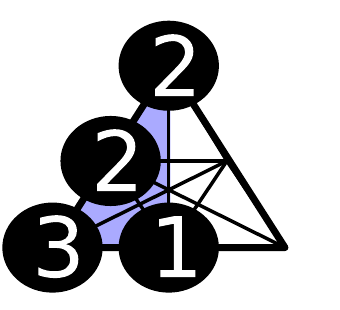} & 0 & 0 & 0 \\
23 & \SimS{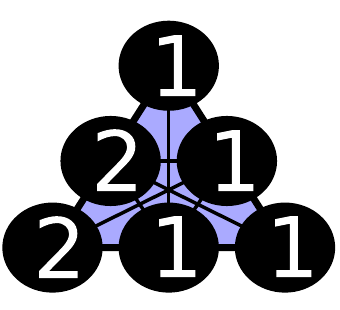} & 0 & 0 & 0 \\
24 & \SimS{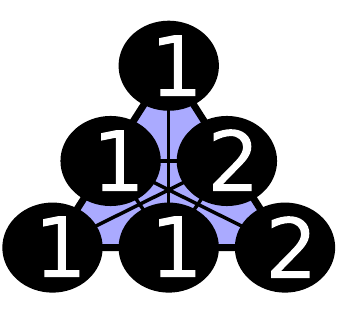} & 0 & 0 & 0 \\
25 & \SimS{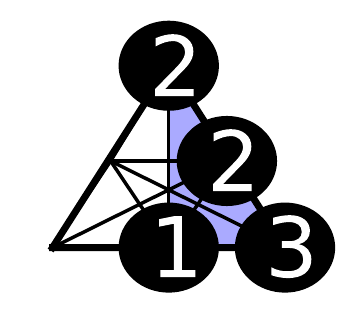} & 0 & 0 & 0 \\ \bottomrule
\end{tabular}
\end{adjustwidth}

\caption{Restriction of the splines $Q_1,\ldots,Q_{25}$, and their directional derivatives, to the boundary edge $e = [\bfv_1, \bfv_2]$.}\label{tab:Qrestrictions1}
\end{table}

\begin{table}
\thisfloatpagestyle{empty}
\begin{adjustwidth}{-0.15\columnwidth}{-0.15\columnwidth}
\vspace{-2.5em}
\centering\begin{tabular}{ccc}
\toprule
$i$  & $Q_i$ & $\frac{1}{3\cdot 4\cdot 5} D^3_{(\alpha_1, \alpha_2, \alpha_3)} Q_i|_e$ \\ \midrule
1 & \SimS{600101} & $32 \alpha_1^3 B_1^2$\\
2 & \SimS{500201} & $96 \alpha_1^2 \alpha_2 B_1^2 + 32\alpha_1^3 B_2^2$ \\
3 & \SimS{410201} & $8\alpha_2^2(6\alpha_1 + \alpha_2) B_1^2 + 4\alpha_2 (12\alpha_1^2 + 6\alpha_1\alpha_2 + \alpha_2^2)B_2^2 + 2(2\alpha_1 + \alpha_2)^3 B_3^2$\\
4 & \SimS{320201} & $8\alpha_2^3 B_1^2 + 8\alpha_2^2(3\alpha_1 + \alpha_2)B_2^2 + 6\alpha_2(2\alpha_1 + \alpha_2)^2 B_3^2 + 4(2\alpha_1 + \alpha_2)^3 B_4^2$\\
5 & \SimS{230210} & $4(\alpha_1 + 2\alpha_2)^3 B_2^2 + 6\alpha_1(\alpha_1 + 2\alpha_2)^2 B_3^2 + 8\alpha_1^2(\alpha_1 + 3\alpha_2)B_4^2 + 8\alpha_1^3 B_5^2$\\
6 & \SimS{140210} & $2(\alpha_1 + 2\alpha_2)^3 B_3^2 + 4\alpha_1(\alpha_1^2 + 6\alpha_1\alpha_2 + 12\alpha_2^2)B_4^2 + 8\alpha_1^2(\alpha_1 + 6\alpha_2)B_5^2$\\
7 & \SimS{050210} & $32\alpha_2^3 B_4^2 + 96\alpha_1\alpha_2^2 B_5^2$\\
8 & \SimS{060110} & $32\alpha_2^3 B_5^2$\\ \midrule
9 & \SimS{500102} & $96\alpha_1^2\alpha_3 B_1^2$\\
10 & \SimS{411101} & $36\alpha_1\alpha_2\alpha_3 B_1^2 + \alpha_3(7\alpha_1^2 + 5\alpha_1\alpha_2+\alpha_2^2) (2 B_2^3 + B_3^2)$\\
11 & \SimS{311201} & $24\alpha_2^2\alpha_3 B_1^2 + 36\alpha_1\alpha_2\alpha_3 B_2^2 + 2\alpha_3(7\alpha_1^2 + 5\alpha_1\alpha_2 + \alpha_2^2)B_3^2$ \\
12 & \SimS{220211} & $8\alpha_3(\alpha_1^2 + 5\alpha_1\alpha_2 + 7\alpha_2^2)B_2^2 + 8\alpha_3(2\alpha_1^2+7\alpha_1\alpha_2+2\alpha_2^2)B_3^2 + 8\alpha_3(7\alpha_1^2 + 5\alpha_1\alpha_2 + \alpha_2^2) B_4^2$\\
13 & \SimS{131210} & $2\alpha_3(\alpha_1^2 + 5\alpha_1\alpha_2 + 7\alpha_2^2)B_3^2 + 36\alpha_1\alpha_2\alpha_3 B_4^2 + 24\alpha_1^2\alpha_3 B_5^2$\\
14 & \SimS{141110} & $\alpha_3(\alpha_1^2 + 5\alpha_1\alpha_2 + 7\alpha_2^2)(B_3^2 + 2B_4^2) + 36\alpha_1\alpha_2\alpha_3 B_5^2$ \\
15 & \SimS{050120} & $96\alpha_2^2\alpha_3 B_5^2$\\ \midrule
16 & \SimS{401102} & $8\alpha_3^2(5\alpha_1 - \alpha_2) B_1^2$\\
17 & \SimS{311102} & $24\alpha_2\alpha_3^2 B_1^2 + 2\alpha_3^2 (5\alpha_1 + 2\alpha_2)(2B_2^2 + B_3^2)$\\
18 & \SimS{211211} & $8\alpha_3^2(\alpha_1 + 4\alpha_2) B_2^2 + 4\alpha_3^2(4\alpha_1 + \alpha_2) B_3^2$\\
19 & \SimS{121211} & $4\alpha_3^2(\alpha_1 + 4\alpha_2) B_3^2 + 8\alpha_3^2(4\alpha_1 + \alpha_2) B_4^2$\\
20 & \SimS{131120} & $2\alpha_3^2(5\alpha_2 + 2\alpha_1)(B_3^2 + 2B_4^2) + 24\alpha_1\alpha_3^2 B_5^2$\\
21 & \SimS{041120} & $8\alpha_3^2(5\alpha_2 - \alpha_1) B_5^2$\\ \midrule
22 & \SimS{302102} & $8\alpha_3^3 B_1^2$\\
23 & \SimS{211112} & $4\alpha_3^3(2B_2^2 + B_3^2)$\\
24 & \SimS{121121} & $4\alpha_3^3(B_3^2 + 2B_4^2)$\\
25 & \SimS{032120} & $8\alpha_3^3 B_5^2$\\ \bottomrule
\end{tabular}
\end{adjustwidth}

\caption{Restriction of the third-order directional derivatives of the splines $Q_1,\ldots,Q_{25}$ to the boundary edge $e = [\bfv_1, \bfv_2]$.}\label{tab:Qrestrictions2}
\end{table}

\subsubsection*{Derivatives on the boundary}
For $d = 5,4,3,2$ and $i = 1,\ldots,d+3$, let $B_i^d$ be the univariate B-spline in Table \ref{tab:shorthandsBi}. Let $(\alpha_1, \alpha_2, \alpha_3)$ be directional coordinates of a vector $\bfu$ with respect to the triangle $[\bfv_1, \bfv_2, \bfv_3]$. Denote by $|_e$ the substitution of $\bfx$ by $(1-t)\bfv_1 + t\bfv_2$. 

In Figure \ref{fig:basis-order} and Tables \ref{tab:Qrestrictions1}, \ref{tab:Qrestrictions2} we order the simplex splines in $\BBB_c$ by the number of knots outside of $e = [\bfv_1, \bfv_2]$. Applying \eqref{eq:MicchelliDifferentiation}, \eqref{eq:knotinsertion}, and \eqref{eq:restriction} we can express, for any simplex spline $Q_i$, the restricted derivative $D_\bfu^k Q_i|_e$ of order $k = 0,1,2,3$ as a linear combination of $B_j^{5-k}$, $j = 1,\ldots,8-k$. These linear combinations are listed in Tables \ref{tab:Qrestrictions1}, \ref{tab:Qrestrictions2} for $Q_1,\ldots, Q_{25}$ and, by \eqref{eq:MicchelliDifferentiation}, are zero for the remaining simplex splines $Q_{26}, \ldots, Q_{39}$.

\begin{example}
We derive the first two entries in the third row in Table \ref{tab:Qrestrictions1}. By~\eqref{eq:restriction}, 
\[ \left.\SimS{410201}\right|_e
 = \frac{\area(\PS)}{\area([\bfv_1, \bfv_2, \bfv_6])} B_3^5
 = 2 B_3^5. \]
Since $\bfu = 2\alpha_1\bfv_1 + 2\alpha_2\bfv_4 + 2\alpha_2\bfv_6$, differentiation and knot insertion yields
\begin{align*}
\left.\frac15 D_\bfu\SimS{410201}\right|_e
& = 2\alpha_1 \left.\SimS{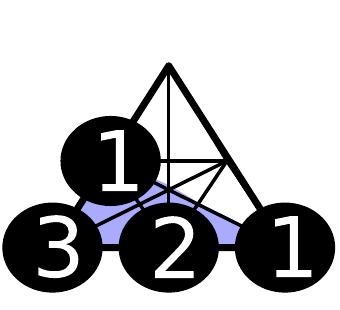}\right|_e + 2\alpha_2 \left.\SimS{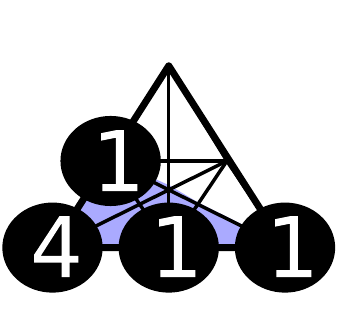}\right|_e + 2 \alpha_3\cdot 0\\
& = 2\alpha_1 \left.\SimS{310201}\right|_e + \alpha_2 \left(\left.\SimS{310201}\right|_e + \left.\SimS{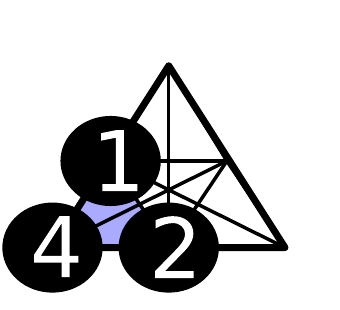}\right|_e\right)\\
& = (2\alpha_1 + \alpha_2)\left.\SimS{310201}\right|_e + \alpha_2\left.\SimS{400201}\right|_e\\
& = 2(2\alpha_1 + \alpha_2) B_3^4 + 4\alpha_2 B_2^4.
\end{align*}
\end{example}

In the next section we apply Tables \ref{tab:Qrestrictions1} and \ref{tab:Qrestrictions2} to derive smoothness conditions for splines on adjacent 12-splits expressed in the basis $\BBB_c$; see Figure~\ref{fig:basis-order}. Analogously, these tables are useful for deriving smoothness conditions with tensor-product, B\'ezier, and more exotic patches.

\subsubsection*{Smoothness conditions}
Let $\PS := [\bfv_1, \bfv_2, \bfv_3]$ and $\tilde{\PS} := [\bfv_1, \bfv_2, \tilde{\bfv}_3]$ be triangles sharing the edge $e := [\bfv_1, \bfv_2]$. Figure \ref{fig:basis-order} shows the domain points of the bases $\BBB_c = \{w_i Q_i\}_i$ and $\tilde{\BBB}_c = \{w_i \tilde{Q}_i\}_i$ of $\SSS^3_5(\PSB)$ and $\SSS^3_5(\tilde{\PSB})$. Let
\begin{equation}\label{eq:neighboringsimplexsplines}
f(\bfv) := \sum_{i=1}^{39} c_i w_i Q_i (\bfv), \ \bfv\in \PS,\qquad
\tilde{f}(\bfv) := \sum_{i=1}^{39} \tilde{c}_i w_i \tilde{Q}_i (\bfv), \ \bfv\in \tilde{\PS}
\end{equation}
be splines defined on these triangles. Imposing a smooth join of $f$ and $\tilde{f}$ along $e$ translates into linear relations among the B\'ezier ordinates $c_i$ and $\tilde{c}_i$.

\begin{theorem}
Let $(\beta_1,\beta_2,\beta_3)$ be the barycentric coordinates of $\tilde{\bfv}_3$ with respect to the triangle $\PS$. Then $f$ and $\tilde{f}$ meet with 

\noindent$C^0$ smoothness if and only if $\tilde{c}_i = c_i$, for $i = 1,\ldots,8$;\\
\noindent$C^1$ smoothness if and only if in addition\\
~\qquad\begin{tabular}{lllll}
&$\tilde{c}_9    = \beta_1c_1 + \beta_2c_2 + \beta_3c_9   $, &
&$\tilde{c}_{15} = \beta_1c_7 + \beta_2c_8 + \beta_3c_{15}$, \\
&$\tilde{c}_{10} = \beta_1c_2 + \beta_2c_3 + \beta_3c_{10}$,  &
&$\tilde{c}_{14} = \beta_1c_6 + \beta_2c_7 + \beta_3c_{14}$, \\
&$\tilde{c}_{11} = \beta_1(2c_3 - c_2) + \beta_2c_4 + \beta_3c_{11}$, &
&$\tilde{c}_{13} = \beta_1c_5 + \beta_2(2c_6 - c_7) + \beta_3c_{13}$, \\
&\multicolumn{2}{l}{$\tilde{c}_{12} = \beta_1\frac{2c_4 + c_5}{3} + \beta_2\frac{c_4 + 2c_5}{3} + \beta_3 c_{12}$;}
\end{tabular}

\noindent$C^2$ smoothness if and only if in addition \\
\begin{tabular}{lll}
& $\tilde{c}_{16} = \beta_1^2 c_1 + 2\beta_1\beta_2 c_2 + \beta_2^2c_3 + 2\beta_1\beta_3 c_9 + 2\beta_2\beta_3 c_{10} + \beta_3^2 c_{16}$,\\

& $\tilde{c}_{17} = \beta_1^2 c_2 + \beta_2^2 c_4 + \beta_3^2c_{17}+ 2\beta_1\beta_2 \frac{3c_3 - c_2}{2} + 2\beta_1\beta_3 \frac{3c_{10} - c_2}{2} + 2\beta_2\beta_3 \frac{c_{10} + 2c_{11} - c_3}{2}$,\\

& $\tilde{c}_{18} = \beta_1^2 \frac{2c_3 + 2c_4 - c_2 }{3} + \beta_2^2 \frac{c_4 + 2c_5}{3} + \beta_3^2c_{18}
+ 2\beta_1\beta_2 \frac{c_2 - 2c_3 + 6 c_4 + c_5}{6} $\\
& \qquad $ + 2\beta_1\beta_3 \frac{c_2 - 2c_3 + 2c_4 - c_5 + 3c_{11} + 3 c_{12}}{6} + 2\beta_2\beta_3 \frac{9c_{12} - 2c_5 - c_{11}}{6}$, \\

& $\tilde{c}_{19} = \beta_1^2 \frac{2c_4 + c_5}{3} + \beta_2^2 \frac{2c_5 + 2c_6 - c_7}{3} +\beta_3^2c_{19}
+ 2\beta_1\beta_2 \frac{c_4 + 6c_5 - 2c_6 + c_7}{6} $\\
& \qquad $+2\beta_2\beta_3 \frac{c_7 - 2c_6 + 2c_5 - c_4 + 3c_{13} + 3c_{12}}{6} +  2\beta_1\beta_3 \frac{9c_{12} -2c_4 - c_{13}}{6} $,\\
 
&$\tilde{c}_{20} = \beta_1^2 c_5 + \beta_2^2 c_7 + \beta_3^2 c_{20} + 2\beta_1\beta_2 \frac{3 c_6 - c_7}{2} + 2\beta_1\beta_3 \frac{c_{14} + 2 c_{13} - c_6}{2} + 2\beta_2 \beta_3 \frac{3 c_{14} - c_7}{2}$, \\

&$\tilde{c}_{21} = \beta_1^2 c_6 + 2\beta_1\beta_2 c_7 + \beta_2^2 c_8 + 2\beta_1\beta_3 c_{14} + 2\beta_2\beta_3 c_{15} + \beta_3^2 c_{21}$;\\
\end{tabular}

\noindent $C^3$ smoothness if and only if in addition \\
\begin{tabular}{lll}

& $ \tilde{c}_{22} = \beta_1^3 c_1 + 3\beta_1^2 \beta_2 \frac{4c_2 - c_1}{3} + 3\beta_1 \beta_2^2 \frac{5c_3 - 2c_2}{3} + \beta_2^3 c_4 + 3\beta_2^2 \beta_3 \frac{c_{10} - c_3 + 3c_{11}}{3} $\\
& \qquad $+ 3\beta_2 \beta_3^2 \frac{c_{10} - c_{16} + 3 c_{17}}{3} + \beta_3^3 c_{22} + 3\beta_3^2 \beta_1 \frac{5c_{16} - 2c_9 }{3} 
+ 3\beta_3 \beta_1^2 \frac{4c_9 - c_1}{3}$\\
& \qquad $+ 6 \beta_1 \beta_2 \beta_3 \frac{5c_{10} - c_2 - c_9}{3}$,\\

& $ \tilde{c}_{23} = \beta_1^3 \frac{c_2 + 2c_3}{3}
    + 3\beta_1^2\beta_2 \frac{- 2c_2 + 8 c_3 + 3 c_4}{9}
    + 3\beta_1\beta_2^2 \frac{c_2 - c_3 + 8 c_4 + c_5}{9}
    + \beta_2^3 \frac{c_4 + 2c_5}{3}$\\
& \qquad $ + 3 \beta_2^2 \beta_3 \frac{c_{10} + 15 c_{12} - c_3 - 2c_4 - 4c_5}{9}
    + 3 \beta_2 \beta_3^2 \frac{-6 c_{12} + 2 c_{17} + 12 c_{18} - c_4 + 2 c_5}{9}$\\
& \qquad $ + 3 \beta_1 \beta_3^2 \frac{c_{10} + 3c_{11} - 3c_{12} + 5c_{17} + 3c_{18} + c_2 - 2c_3 + c_5}{9}
    + 3 \beta_1^2 \beta_3 \frac{8 c_{10} + 3 c_{11} - 2c_2}{9}$\\
& \qquad $ + 6 \beta_1 \beta_2 \beta_3 \frac{3c_{10} + 6 c_{11} + 3 c_{12} + c_2 - 4 c_3 + c_4 - c_5}{9}
    +   \beta_3^3 c_{23}$,\\

& $ \tilde{c}_{24} = \beta_1^3 \frac{2c_4 + c_5}{3}
+ 3 \beta_1^2 \beta_2 \frac{c_4 + 8c_5 - c_6 + c_7}{9}
+ 3 \beta_1 \beta_2^2 \frac{3c_5 + 8c_6 - 2c_7}{9}
+ \beta_2^3 \frac{2c_6 + c_7}{3}$\\
& \qquad $
+ 3 \beta_2^2 \beta_3 \frac{3c_{13} + 8c_{14} - 2c_7}{9}
+ 3 \beta_2 \beta_3^2 \frac{-3c_{12} + 3 c_{13} + c_{14} + 3 c_{19} + 5 c_{20} + c_4 - 2 c_6 + c_7}{9}$\\
& \qquad $
+ 3 \beta_1 \beta_3^2 \frac{-6 c_{12} + 12 c_{19} + 2 c_{20} + 2 c_4 - c_5}{9}
+ 3 \beta_1^2 \beta_3 \frac{15 c_{12} + c_{14} - 4 c_4 - 2 c_5 - c_6}{9}$\\
& \qquad $
+ 6 \beta_1 \beta_2 \beta_3 \frac{3 c_{12} + 6 c_{13} + 3 c_{14} - c_4 + c_5 - 4 c_6 + c_7}{9} + \beta_3^3 c_{24}$,\\

& $ \tilde{c}_{25} = \beta_1^3 c_5 + 3\beta_1^2 \beta_2 \frac{5c_6 - 2c_7}{3} + 3 \beta_1 \beta_2^2 \frac{4c_7 - c_8}{3} + \beta_2^3 c_8 + 3\beta_2^2 \beta_3 \frac{4c_{15} - c_8}{3}$\\
& \qquad $ + 3\beta_2 \beta_3^2 \frac{5c_{21} - 2c_{15}}{3} + 3\beta_3^2 \beta_1 \frac{c_{14} + 3c_{20} - c_{21}}{3} 
 + 3\beta_3 \beta_1^2 \frac{3c_{13} + c_{14} - c_6}{3} $ \\
& \qquad $+ 6\beta_1\beta_2\beta_3 \frac{5 c_{14} - c_7 - c_{15}}{3} + \beta_3^3 c_{25}$, \\

& $\phantom{+} 0 = (3\beta_1^2\beta_2 - 3\beta_1\beta_2^2) \frac{c_2 - 2 c_3 + 2 c_4 - 2 c_5 + 2 c_6 - c_7}{3}$\\
& \qquad $+ 3\beta_1^2 \beta_3 \frac{c_{11} - 3c_{12} + c_{13} + c_2 - 2 c_3 + 2 c_4}{3}$\\
& \qquad $+ 3\beta_2^2 \beta_3 \frac{c_{11} - 3c_{12} + c_{13} + c_7 - 2 c_6 + 2 c_5}{3}$\\
& \qquad $+ 3\beta_1 \beta_3^2 \frac{-5 c_{11} + 6 c_{12} + c_{13} + 9 c_{18} - 9 c_{19} - 2c_2 + 4c_3 - 4c_4}{6}$\\
& \qquad $+ 3\beta_2 \beta_3^2 \frac{-5 c_{13} + 6 c_{12} + c_{11} + 9 c_{19} - 9 c_{18} - 2c_7 + 4c_6 - 4c_5}{6}$\\
& \qquad $+ 6\beta_1 \beta_2 \beta_3 \frac{-2c_{11} + 6 c_{12} - 2 c_{13} - c_2 + 2 c_3 - 2 c_4 - 2 c_5 + 2 c_6 - c_7}{3}$.\\

\end{tabular}
\end{theorem}

\begin{proof}
By the barycentric nature of the statement, we can change coordinates by the linear affine map that sends $\bfv_1\longmapsto (0,0)$, $\bfv_2\longmapsto (1,0)$, and $\bfv_3\longmapsto (0,1)$. In these coordinates,
\[ \tilde{\bfv}_3 = \beta_1(0,0) + \beta_2(1,0) + \beta_3(0,1) = (\beta_2, \beta_3).\]
Let $\bfu := \tilde{\bfv}_3 - \bfv_1 = (\beta_2, \beta_3)$. For $r = 0,1,2,3$, the splines $f$ and $\tilde{f}$ meet with $C^r$ smoothness along $e$ if and only if $D^k_\bfu f(\cdot,0) = D^k_\bfu \tilde{f}(\cdot,0)$ for $k = 0, \ldots, r$. Substituting \eqref{eq:neighboringsimplexsplines}, this is equivalent to
\[ \sum_{i=1}^{39} c_i w_i D^k_\bfu Q_i(\cdot,0) = \sum_{i=1}^{39} \tilde{c}_i w_i D^k_\bfu \tilde{Q}_i(\cdot,0),\qquad k = 0,\ldots, r, \]
which, using Tables \ref{tab:Qrestrictions1} and \ref{tab:Qrestrictions2}, reduces to a sparse system \[ \sum_{j=1}^{8-k} r_{kj} B_j^{5-k} = 0,\qquad k = 0,\ldots, r,\]
where $r_{kj}$ is a linear combination of $c_i$ and $\tilde{c}_i$ with $i = 1,\ldots, n_{k+1}$, where $n_1 = 8, n_2 = 15, n_3 = 21$, and $n_4 = 25$. This system holds identically if and only if $r_{kj} = 0$ for $j = 1,\ldots,8-k$ and $k = 0,\ldots,r$. Let $n_0 = 0$. For $k = 0, 1, 2, 3$, one solves for $\tilde{c}_{n_k + 1}, \ldots, \tilde{c}_{n_{k+1}}$, each time eliminating the B\'ezier ordinates $\tilde{c}_i$ that were previously obtained, resulting in the smoothness relations of the Theorem; see the worksheet for details \cite{WebsiteGeorg}.
\end{proof}

As for the B\'ezier basis and the S-basis from \cite{Cohen.Lyche.Riesenfeld13}, each smoothness relation also holds when replacing each B\'ezier ordinate by the corresponding domain point. The smoothness relations therefore also hold between the corresponding control points.

The final $C^3$ smoothness condition only involves the B\'ezier ordinates in a single triangle and the barycentric coordinates of the opposing vertex in a neighboring triangle. It follows that $C^3$ smoothness using $\SSS^3_5$ cannot be achieved on a general refined triangulation $\TTT_{12}$.
	
\subsubsection*{Conversion to Hermite nodal basis}
Let $\Lambda$ be as in \eqref{eq:Lambda} with $\bfx_{\bfv} := \bfv_i - \bfv$ and $\bfy_{\bfv} := \bfv_j - \bfv$ for any vertex $\bfv \in \VVV$ with opposing edge $[\bfv_i, \bfv_j]\in \EEE$, and $\bfu_e = \bfv_k - \bfm_e$ for any edge $e$ with opposing vertex $\bfv_k$. In addition to the basis $\BBB_c = \{w_i Q_i\}_{i=1}^{39}$, the spline space $\SSS^3_5(\PSB)$ has the (Hermite) nodal basis $\Lambda^* = \{\lambda_i^*\}_{i=1}^{39}$ dual to $\Lambda = \{\lambda_i\}_{i=1}^{39}$, i.e., $\lambda_j(\lambda_i^*) = \delta_{ij}$. In this section we express $\Lambda^*$ in terms of $\BBB_c$. For details we refer to the worksheet~\cite{WebsiteGeorg}.

Write $Q_i = a_{i, 1} \lambda_1^* + \cdots + a_{i,39} \lambda_{39}^*$ for $i = 1, \ldots, 39$, so that $\lambda_j (Q_i) = a_{ij}$. Multiplying by the inverse of the matrix $(a_{ij})_{i,j}$, we can express the nodal basis functions $\lambda_1^*, \ldots \lambda_{39}^*$ in terms of $Q_1, \ldots, Q_{39}$.

\allowdisplaybreaks
\begin{theorem}\label{thm:conversion}
With $\bfv = \bfv_1$, $\bfx = \bfv_2 - \bfv_1$, $\bfy = \bfv_3 - \bfv_1$, and $\bfu = \bfv_3 - \bfv_4$,
\begin{align*}
\varepsilon_{\bfv}^*      = & \frac14 \SimS{600101} +  \frac14 \left(\SimS{500201} + \SimS{500102}\right) + \frac12 \left(\SimS{410201} + \SimS{401102}\right) \\
&  +  \SimS{411101} +  \frac12 \left(\SimS{311201} + \SimS{311102}\right) + \frac12 \left(\SimS{320201} + \SimS{302102}\right) \\
& + \frac{9}{16} \left(\SimS{211211} + \SimS{211112}\right) + \frac38 \left(\SimS{220211} + \SimS{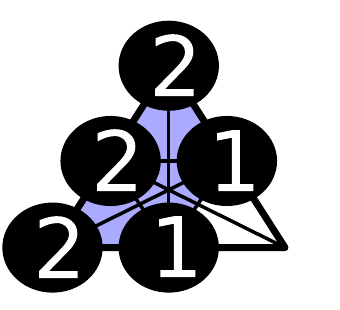}\right)\\
& + \frac{3}{16} \left(\SimS{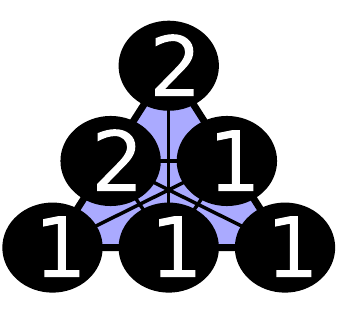} + \SimS{121211}\right)  \\
(\varepsilon_{\bfv} D_{\bfx})^* = & \frac{1}{40} \SimS{500201} + \frac{1}{10} \SimS{410201} + \frac15 \SimS{320201} + \frac{1}{10} \SimS{411101} \\
&  + \frac{3}{20} \SimS{311201} + \frac{13}{80} \SimS{220211} + \frac{1}{20} \SimS{311102} + \frac{19}{80} \SimS{211211}\\
&  + \frac{1}{10} \SimS{121211} - \frac{1}{40} \SimS{211112} - \frac{1}{40} \SimS{202112} - \frac{1}{20} \SimS{112112} \\
(\varepsilon_{\bfv} D^2_{\bfx})^* = & \frac{1}{160} \SimS{410201} + \frac{1}{32} \SimS{320201} + \frac{1}{80} \SimS{311201} + \frac{17}{640} \SimS{220211}\\
& + \frac{121}{3840} \SimS{211211} + \frac{71}{3840} \SimS{121211} - \frac{1}{48} \SimS{211112} + \frac{1}{480} \SimS{112112} \\
(\varepsilon_{\bfv} D_{\bfx} D_{\bfy})^*  = & \frac{1}{80} \SimS{411101} + \frac{3}{160} \left(\SimS{311201} + \SimS{311102}\right) - \frac{1}{160}\left(\SimS{220211} + \SimS{202112}\right) \\
& + \frac{17}{960} \left(\SimS{211112}  + \SimS{211211}\right) - \frac{7}{480} \left(\SimS{112112} + \SimS{121211}\right)\\
(\varepsilon_{\bfv} D^3_{\bfx})^* = & \frac{1}{480} \SimS{320201} + \frac{7}{3840} \SimS{220211} + \frac{5}{3072} \SimS{211211} + \frac{7}{5120} \SimS{121211}\\
(\varepsilon_{\bfv} D^2_{\bfx} D_{\bfy})^*  = & \frac{1}{480} \SimS{311201} - \frac{1}{1920} \SimS{220211} - \frac{1}{768} \SimS{121211}\\
& + \frac{11}{3840} \SimS{211211} - \frac{11}{3840} \SimS{211112} + \frac{1}{3840} \SimS{112112}\\
(\varepsilon_{\bfq_{1,e}} D^2_{\bfu})^* = & \frac{7}{240} \SimS{211211} - \frac{1}{240} \SimS{121211} \\
(\varepsilon_{\bfm_e} D_{\bfu})^* = & \frac{1}{10} \SimS{220211} + \frac15 \left(\SimS{211211} + \SimS{121211}\right).
\end{align*}
\end{theorem}

Note that the coefficients in these linear combinations are independent of the geometry of the triangle. The remaining nodal functions in $\Lambda^*$ are obtained by applying a symmetry in $S_3$ to the above equations. For instance, 
\[ (\varepsilon_{\bfq_{2,e}} D^2_{\bfu})^* = \frac{7}{240} \SimS{121211} - \frac{1}{240} \SimS{211211}. \]

\begin{example}\label{ex:NumericalExample}
Let $\bfv_i := \big(\cos(2\pi i/6), \sin(2\pi i/6) \big)$, with $i = 1,2, \ldots, 6$, be the vertices of a regular hexagon centred at the origin $\bfv_0 := (0, 0)$. Consider the triangulation $\TTT$ with triangles 
$[\bfv_0, \bfv_6, \bfv_1], [\bfv_0, \bfv_1, \bfv_2], \ldots, [\bfv_0, \bfv_5, \bfv_6]$. The nodal functions $\varepsilon_{\bfv_0}^*$ on these triangles patch together to a spline in $\SSS^{2,3}_5(\TTT_{12})$, which is plotted in Figure~\ref{fig:NodalFunction} together with its control mesh.
\end{example}

\begin{figure}
\begin{center}
\includegraphics[scale = 0.11]{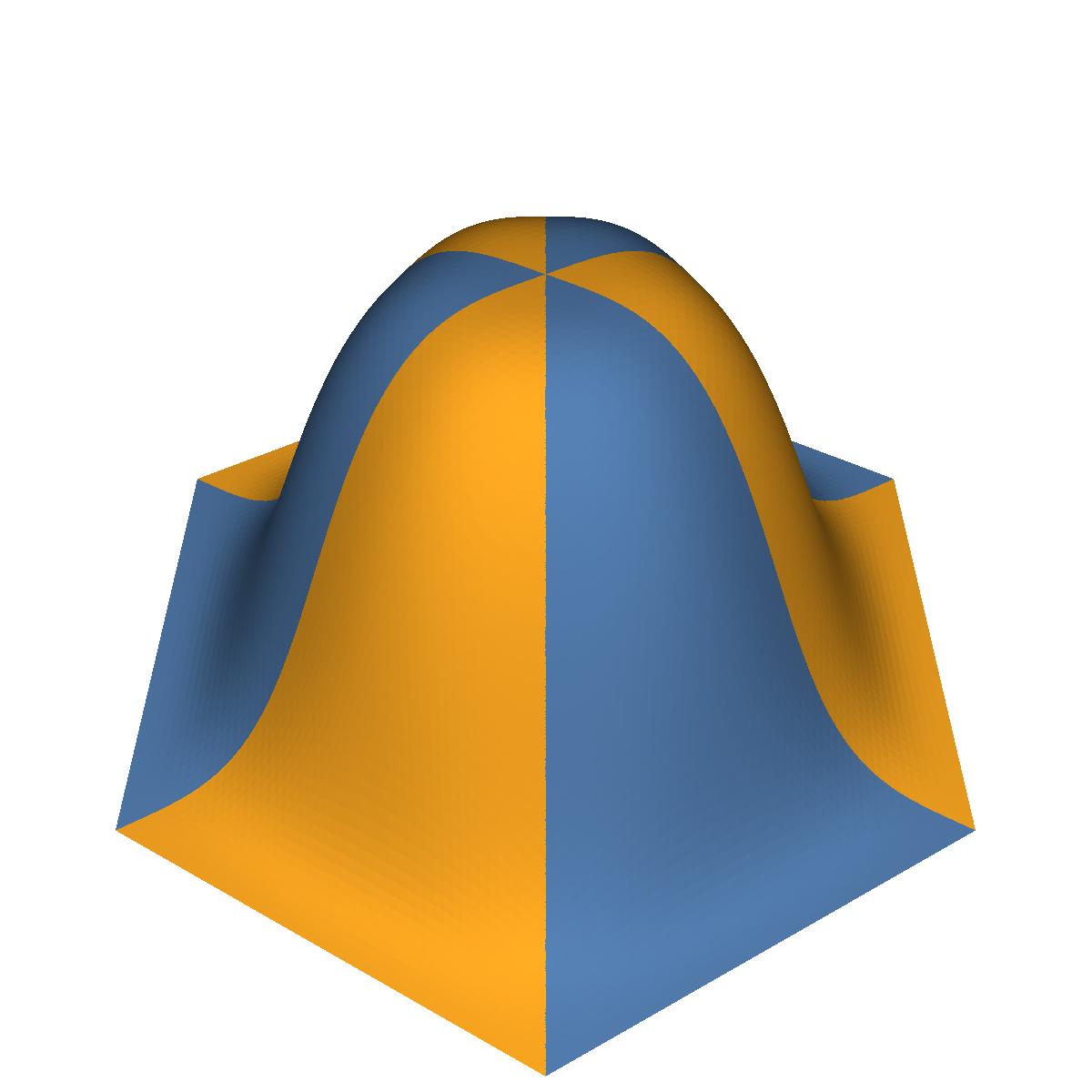}
\includegraphics[scale = 0.11]{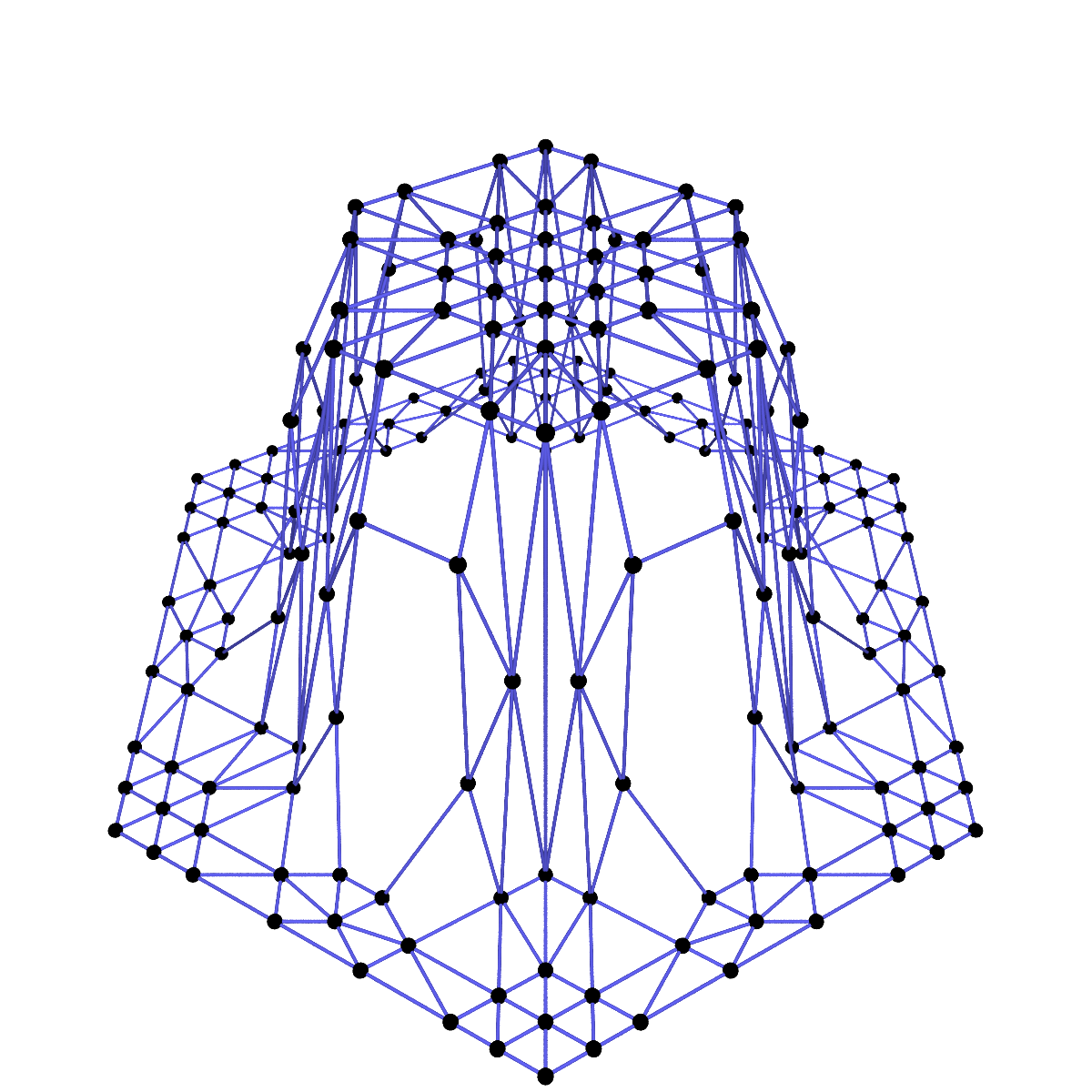}
\end{center}
\caption{The nodal function $\varepsilon_{\bfv_0}^*$ (left) and its control mesh (right) on the triangulation from Example \ref{ex:NumericalExample}.}\label{fig:NodalFunction}
\end{figure}

\section{Final remarks}

\begin{remark}
On the Powell-Sabin 6-split a simplex spline basis is not possible, because it is not a complete graph. However,  there exist $C^2$ quintic B-spline bases \cite{Alfeld.Schumaker02, Speleers10, Speleers13}.
\end{remark}

\begin{remark}
The $C^1$ quadratics from \cite{Powell.Sabin77} and the space from \cite{Lyche.Muntingh14} can be viewed as the cases $n=1, 2$ of a sequence of locally $C^{2n-1}$, globally $C^n$ spaces of degree $3n-1$. There is a natural generalization to general $n$ of a set of nodal functionals that on a single triangle has size equal to the dimension $\frac{15}{2}n^2 + \frac92 n$ of this space. As remarked in \cite{Lyche.Muntingh14}, these do not form a basis for $n > 2$. However, it is plausible that one can instead construct simplex spline bases for higher degree and smoothness.
\end{remark}

\begin{remark}
A case-by-case analysis, and a computation in the worksheet, show that the $C^1$ quadratic simplex splines on the 12-split are
\[ \left[
\SimS{300101},\
\SimS{210101},\
\SimS{110111},\
\SimS{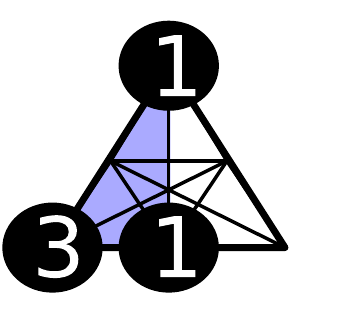},\
\SimS{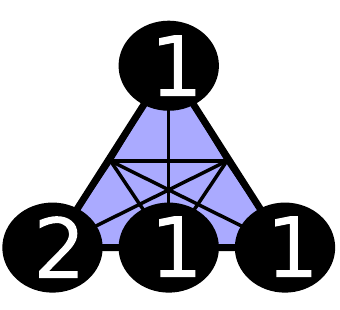},\
\SimS{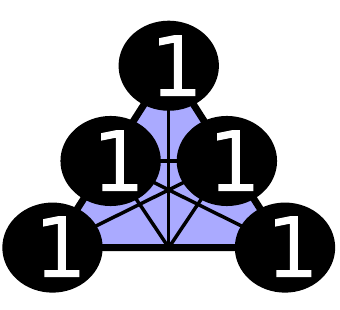},\
\SimS{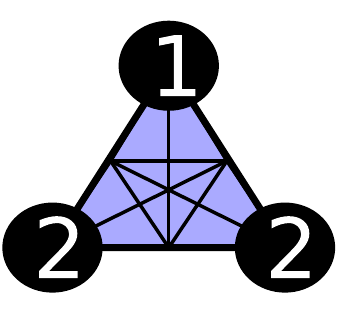},\
\SimS{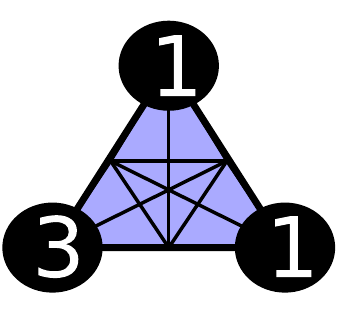}
\right]_{S_3}. \]
The quadratic S-basis of $\SSS^1_2$ comprises the first three types and was shown to have local linear independence. Taking any other combination will overload some of the triangles in the 12-split, and the S-basis is therefore the unique simplex spline basis with local linear independence. Moreover, it reduces to a B-spline basis on the boundary.

For the quintic splines, a case-by-case analysis shows that each of the outer faces $\PS_1, \PS_2, \ldots, \PS_6$ will be covered by at least 9 of the 21 simplex splines that are nonzero on the boundary, and by at least 14 of the 18 remaining simplex splines. Thus each outer face is covered by at least $23 > (5+1)(5+2)/2$ simplex splines, showing that $\SSS^3_5$ admits no locally linearly independent $S_3$-invariant simplex spline basis that reduces to a B-spline basis on the boundary.
\end{remark}


\begin{bibdiv}
\begin{biblist}

\bib{Alfeld.Schumaker02}{article}{
   author={Alfeld, Peter},
   author={Schumaker, Larry L.},
   title={Smooth macro-elements based on Powell-Sabin triangle splits},
   journal={Adv. Comput. Math.},
   volume={16},
   date={2002},
   number={1},
   pages={29--46},
   issn={1019-7168},
}

\bib{Chui.Wang83}{article}{
   author={Chui, Charles K.},
   author={Wang, Ren Hong},
   title={Multivariate spline spaces},
   journal={J. Math. Anal. Appl.},
   volume={94},
   date={1983},
   number={1},
   pages={197--221},
   issn={0022-247X},
}

\bib{Ciarlet.78}{book}{
   author={Ciarlet, Philippe G.},
   title={The finite element method for elliptic problems},
   series={Classics in Applied Mathematics},
   volume={40},
   note={Reprint of the 1978 original [North-Holland, Amsterdam]},
   publisher={Society for Industrial and Applied Mathematics (SIAM),
   Philadelphia, PA},
   date={2002},
   pages={xxviii+530},
   isbn={0-89871-514-8},
}

\bib{Cohen.Riensenfeld.Elber01}{book}{
   author={Cohen, Elaine},
   author={Riesenfeld, Richard F.},
   author={Elber, Gershon},
   title={Geometric modeling with splines},
   note={An introduction;
   With a foreword by Tom Lyche},
   publisher={A K Peters, Ltd., Natick, MA},
   date={2001},
   pages={xxii+616},
}

\bib{Cohen.Lyche.Riesenfeld13}{article}{
   author={Cohen, Elaine},
   author={Lyche, Tom},
   author={Riesenfeld, Richard F.},
   title={A B-spline-like basis for the Powell-Sabin 12-split based on
   simplex splines},
   journal={Math. Comp.},
   volume={82},
   date={2013},
   number={283},
   pages={1667--1707},
   issn={0025-5718},
}

\bib{Davydov.Yeo13}{article}{
   author={Davydov, Oleg},
   author={Yeo, Wee Ping},
   title={Refinable $C^2$ piecewise quintic polynomials on
   Powell-Sabin-12 triangulations},
   journal={J. Comput. Appl. Math.},
   volume={240},
   date={2013},
   pages={62--73},
   issn={0377-0427},
}


\bib{DynLyche98}{article}{
   author={Dyn, Nira},
   author={Lyche, Tom},
   title={A Hermite subdivision scheme for the evaluation of the
   Powell-Sabin $12$-split element},
   conference={
      title={Approximation theory IX, Vol. 2},
      address={Nashville, TN},
      date={1998},
   },
   book={
      editor={Charles K. Chui},
      editor={Larry L. Schumaker},
      series={Innov. Appl. Math.},
      publisher={Vanderbilt Univ. Press},
      place={Nashville, TN},
   },
   date={1998},
   pages={33--38},
}

\bib{Goodman.Lee81}{article}{
   author={Goodman, T. N. T.},
   author={Lee, S. L.},
   title={Spline approximation operators of Bernstein-Schoenberg type in one
   and two variables},
   journal={J. Approx. Theory},
   volume={33},
   date={1981},
   number={3},
   pages={248--263},
   issn={0021-9045},
}

\bib{Hollig82}{article}{
   author={H{\"o}llig, Klaus},
   title={Multivariate splines},
   journal={SIAM J. Numer. Anal.},
   volume={19},
   date={1982},
   number={5},
   pages={1013--1031},
   issn={0036-1429},
}

\bib{Hughesbook}{book}{
   author={Cottrell, J. Austin},
   author={Hughes, Thomas J.R.},
   author={Bazilevs, Yuri},
   title={Isogeometric analysis: toward integration of CAD and FEA},
   publisher={Wiley Publishing},
   date={August 2009},
   pages={360},
}

\bib{Knuth94}{article}{
   author={Knuth, Donald},
   title={Bracket notation for the ``coefficient of'' operator},
   eprint={http://arxiv.org/abs/math/9402216},
   book={
      editor = {Roscoe, A. W.},
      title = {A Classical Mind: Essays in Honour of C. A. R. Hoare},
      year = {1994},
      isbn = {0-13-294844-3},
      address = {Hertfordshire, UK},
      publisher = {Prentice Hall International (UK) Ltd.},
   },
}

\bib{Lai.Schumaker03}{article}{
   author={Lai, Ming-Jun},
   author={Schumaker, Larry L.},
   title={Macro-elements and stable local bases for splines on Powell-Sabin
   triangulations},
   journal={Math. Comp.},
   volume={72},
   date={2003},
   number={241},
   pages={335--354},
   issn={0025-5718},
}

\bib{Lai.Schumaker07}{book}{
   author={Lai, Ming-Jun},
   author={Schumaker, Larry L.},
   title={Spline functions on triangulations},
   series={Encyclopedia of Mathematics and its Applications},
   volume={110},
   publisher={Cambridge University Press},
   place={Cambridge},
   date={2007},
   pages={xvi+592},
   isbn={978-0-521-87592-9},
   isbn={0-521-87592-7},
}

\bib{Lyche.Muntingh14}{article}{
   author={Lyche, Tom},
   author={Muntingh, Georg},
   title={A Hermite interpolatory subdivision scheme for $C^2$-quintics
   on the Powell-Sabin 12-split},
   journal={Comput. Aided Geom. Design},
   volume={31},
   date={2014},
   number={7-8},
   pages={464--474},
   issn={0167-8396},
}

\bib{Micchelli79}{article}{
   author={Micchelli, Charles A.},
   title={On a numerically efficient method for computing multivariate
   $B$-splines},
   conference={
      title={Multivariate approximation theory},
      address={Proc. Conf., Math. Res. Inst., Oberwolfach},
      date={1979},
   },
   book={
	  editor = {Walter Schempp},
	  editor = {Karl Zeller},	  
      series={Internat. Ser. Numer. Math.},
      volume={51},
      publisher={Birkh\"auser, Basel-Boston, Mass.},
   },
   date={1979},
   pages={211--248},
}

\bib{WebsiteGeorg}{article}{
   author={Muntingh, Georg},
   title={Personal Website},
   eprint={https://sites.google.com/site/georgmuntingh/academics/software}
}

\bib{Oswald92}{article}{
   author={Oswald, Peter},
   title={Hierarchical conforming finite element methods for the biharmonic
   equation},
   journal={SIAM J. Numer. Anal.},
   volume={29},
   date={1992},
   number={6},
   pages={1610--1625},
   issn={0036-1429},
}

\bib{Powell.Sabin77}{article}{
   author={Powell, Michael J. D.},
   author={Sabin, Malcolm A.},
   title={Piecewise quadratic approximations on triangles},
   journal={ACM Trans. Math. Software},
   volume={3},
   date={1977},
   number={4},
   pages={316--325},
   issn={0098-3500},
}

\bib{Prautsch.Boehm.Paluszny02}{book}{
   author={Prautzsch, Hartmut},
   author={Boehm, Wolfgang},
   author={Paluszny, Marco},
   title={B\'ezier and B-spline techniques},
   series={Mathematics and Visualization},
   publisher={Springer-Verlag, Berlin},
   date={2002},
   pages={xiv+304},
   isbn={3-540-43761-4},
}

\bib{Sage}{book}{
   author={Stein, William A.},
   author={others},
   organization = {The Sage Development Team},
   title = {{S}age {M}athematics {S}oftware ({V}ersion 6.5)},
   eprint = {{\tt http://www.sagemath.org}},
   date = {2015},
}

\bib{Schenck.Stillman97}{article}{
   author={Schenck, Hal},
   author={Stillman, Mike},
   title={A family of ideals of minimal regularity and the Hilbert series of
   $C^r(\hat \Delta)$},
   journal={Adv. in Appl. Math.},
   volume={19},
   date={1997},
   number={2},
   pages={169--182},
   issn={0196-8858},
}

\bib{Schumaker.Sorokina06}{article}{
   author={Schumaker, Larry L.},
   author={Sorokina, Tatyana},
   title={Smooth macro-elements on Powell-Sabin-12 splits},
   journal={Math. Comp.},
   volume={75},
   date={2006},
   number={254},
   pages={711--726 (electronic)},
   issn={0025-5718},
}

\bib{Seidel92}{article}{
   author={Seidel, Hans-Peter},
   title={Polar forms and triangular $B$-spline surfaces},
   conference={
      title={Computing in Euclidean geometry},
   },
   book={
      series={Lecture Notes Ser. Comput.},
      volume={1},
      publisher={World Sci. Publ., River Edge, NJ},
   },
   date={1992},
   pages={235--286},
}

\bib{Speleers10}{article}{
   author={Speleers, Hendrik},
   title={A normalized basis for quintic Powell-Sabin splines},
   journal={Comput. Aided Geom. Design},
   volume={27},
   date={2010},
   number={6},
   pages={438--457},
   issn={0167-8396},
}

\bib{Speleers13}{article}{
   author={Speleers, Hendrik},
   title={Construction of normalized B-splines for a family of smooth spline
   spaces over Powell-Sabin triangulations},
   journal={Constr. Approx.},
   volume={37},
   date={2013},
   number={1},
   pages={41--72},
   issn={0176-4276},
}

\bib{Zenisek.74}{article}{
   author={{\v{Z}}en{\'{\i}}{\v{s}}ek, Alexander},
   title={A general theorem on triangular finite $C^{(m)}$-elements},
   language={English, with Loose French summary},
   journal={Rev. Fran\c caise Automat. Informat. Recherche Op\'erationnelle
   S\'er. Rouge},
   volume={8},
   date={1974},
   number={R-2},
   pages={119--127},
}

\end{biblist}
\end{bibdiv}
\end{document}